\documentclass[11pt,a4paper]{article}

\usepackage{amsmath, amsfonts, amssymb}
\usepackage{color}

\def\be#1\ee{\begin{equation}#1\end{equation}}

\newtheorem{thm}{Theorem}
\newtheorem{lem}[thm]{Lemma}
\newtheorem{prop}[thm]{Proposition}

\newtheorem{rem}[thm]{Remark}

\DeclareMathOperator{\diag}{diag}

\DeclareMathOperator{\var}{var}

\DeclareMathOperator{\cov}{cov}

\def\P{{\mathbb{P}}}
\def\R{\mathbb{R}}
\def\E{\mathbb{E}\,}

\def\dd{\mbox{d}}

\newenvironment{proof}[1][] {\noindent {\bf Proof#1:} }{\hspace*{\fill}$\square$\medskip\par}

\def\DD{{\mathcal D}}

\newcommand{\eps}{\varepsilon}

\def\TT{{\mathcal{T}}}
\def\U{\mbox{{P}}}

\def\WW{{\mathcal{W}}}
\def\vecWW{\boldsymbol{\WW}}
\def\XX{{\mathcal{X}}}
\def\vecX{\boldsymbol{X}}
\def\vecXX{\boldsymbol{\XX}}
\def\ZZZ{{\mathfrak{Z}}}
\def\vecZ{\boldsymbol{Z}}

\def\vecpsi{\boldsymbol{\psi}}

\def\tvecX{\widetilde{\boldsymbol{X}}}
\def\deq{{~\stackrel{d}{=}~}}

\def\MA{\AA} 

\def\newY{{\mathfrak{Y}}}
\def\ZZZ{{\mathfrak{Z}}}
\def\vecZZZ{\boldsymbol{{\mathfrak{Z}}}}

\def\generalizedq{{g}}
\def\AA{\mathcal{A}}

\let\BFseries\bfseries\def\bfseries{\BFseries\mathversion{bold}} 

\def\tod{{~\stackrel{d}{\to}~}}     

\def\uu{{\mathfrak u}}
\def\br{{\mathfrak b}}
\def\hX{\widehat{X}}
\def\tX{\widetilde{X}}
\def\teps{\widetilde{\eps}}
\def\tsigma{\widetilde{\sigma}}
\def\hB{\widehat{B}}

\begin{document}

\title{Universal break law for chains of Brownian particles
\\ with nearest neighbour interaction}
\author{Frank Aurzada, Volker Betz, and Mikhail Lifshits}
\maketitle

\begin{abstract}
We investigate the behaviour of a finite chain of Brownian particles,
interacting through a pairwise potential $U$, with one end of the chain fixed and
the other end  pulled away, in the limit of slow pulling speed and small Brownian noise.
We study the instant when and the place where the chain ``breaks'', that is,
the distance between two neighbouring particles becomes larger than a certain threshold.

We assume $U$ to be attractive and strictly convex up to the break distance, and three times continuously differentiable.
We consider the regime, where both the pulling and the noise
significantly influence the distribution of the break time and break
position. It turns out that in this regime there is a universality
of both the break time distribution and the break position distribution,
in the sense that the limiting quantities do not depend on the details of $U$,
but only on its curvature at the break distance.
\end{abstract}

\medskip


\section{Introduction}
Interacting Brownian particles are a natural and popular model for physical systems such as
crystals, soft matter, or interacting colloidal particles: the interaction models the force between the particles, while the
noise models external influences on the system, such as the collision with much smaller
particles that are not explicitly modelled, or thermal fluctuations. Some of the
possible application scenarios are given e.g.\ in \cite{Sei12,SKF2016}.

If systems of interacting particles evolve in a smooth way (on large scales),
it is natural to investigate their macroscopic behaviour via hydrodynamic limits.
This is by no means an easy problem, but good progress has been made over the last decades.
For systems at equilibrium, important results include the seminal work \cite{Sp86} on Gaussian fluctuations around reversible equilibria,
and the recent significant progress \cite{DGP17} on convergence to the KPZ equation for weakly asymmetric one-dimensional
Ginzburg-Landau interface models. For the hydrodynamic limit of general non-equilibrium systems, the classical result \cite{Va90}
establishes the hydrodynamic limit for one-dimensional systems with repulsive interactions on the torus.
In the case of Ginzburg-Landau models, significant progress (in particular, extension to higher dimensions) was achieved in
\cite{FS97}.

A very different situation arises in cases where the system of interacting particles behaves in a way that is discontinuous
on the macroscopic scale, such as when a material breaks under strain. The most significant, but also the most difficult instance
of such a situation is the dynamics of the propagation of cracks through solids, which occurs on a different time scale than the standard macroscopic
dynamics. Due to this fact there are, to our knowledge, no mathematical tools for rigorously investigating the problem, and all the
activity is on numerical studies, see e.g.\ \cite{Ba17,PH19}. Even for the one-dimensional case, i.e.\ the rupture of a
molecular chain under strain, the majority of the activity is non-rigorous, such as \cite{FS1, FS2, lee, RBM19}.

Indeed, one of very few\footnote{There is another model represented in \cite{Muz}; see also \cite{MalMuz,Mal}.}
mathematically rigorous attempts on the problem of chain rupture that we are aware of
is that started in \cite{allman_betz_09}, and then extended in different ways in \cite{ABL, allman_betz_11}.
In the present paper, we significantly advance the understanding of that model; the most important new observation is that in the
most relevant parameter regime (the intermediate regime), for a chain of finite but arbitrary length,
the asymptotic distribution of both the break time and position are universal in the sense that they do not depend on most details of the intermolecular force.

Our model is mathematically
equivalent to the one-dimensional Ginz\-burg-Lan\-dau model with a time-dependent boundary condition.
To be precise, let $d\geq 2$ be an integer, and consider a chain of $d+1$ particles located
on the real line, interacting via a nearest neighbour\footnote{We always use the word `neighbour' to mean consecutive indices, not neighbours in space.
This is justified e.g.\ if we think of the chain as consisting of a string of molecules. On the other hand, in the
applications we have in mind particles swap places with negligible probability, in which case there is no difference between spatial
and index neighbours.}
force given by the derivative of a potential $U$.
The positions of the particles at time $t$ are denoted by
$\XX^0_t,\dots, \XX^d_t$. We assume that $\XX_0^i = i$, $\XX^0_t = 0$ for all $t$, and $\XX^d_t = \XX^d_0 + \eps t$ for $\eps>0$ and all $t$.
This means that the leftmost particle is fixed and the rightmost one is pulled with speed $\eps$ to the right.
Altogether, the model is thus described by the system of stochastic differential equations
\begin{equation} \label{eqn:original_U}
   \begin{cases}
     \XX_0^i = i & i=0,1,\ldots, d;\\
     \XX_t^0 = 0 & t\geq 0;\\
     \XX_t^d = d+\eps t & t\geq 0;\\
     \dd \XX_t^i = \left(U'(\XX^{i+1}_t-\XX^{i}_t) - U'(\XX^{i}_t- \XX^{i-1}_t)\right)\dd t
     +\sigma \dd B_t^i,& i=1,\ldots,d-1, t\geq 0,
   \end{cases}
\end{equation}
where $(B_t^i)_{t\geq 0}$ are independent Brownian motions, $i=1,\ldots,d-1$,
$\sigma\geq 0$, $\eps\geq 0$, and $U$ is a sufficiently regular function. We will often write
${\vecXX}_t:=(\XX^0_t,\dots,\XX^d_t)^\top$.

We will be interested the asymptotic behaviour of the model as $\eps$ and $\sigma$ vanish. 
Quantities of interest are the time and location (along the chain) of the chain rupture under the dynamics.
The physically most desirable choice for $U$
is a potential that is attractive at short distances but becomes flat at infinity. This then
leads to a motion where at first the chain becomes more and more elongated, until at some point a fluctuation makes one of the
gaps between two neighbouring particles so large that it energetically favourable for the chain to split into two disconnected
pieces. It is not hard to see that e.g.\ for potentials $U$ that are
strictly convex on an interval containing the starting distance of $1$ between two particles,
this critical gap size corresponds to the first inflection point of $U$ on $[1,\infty)$,
i.e.\ the smallest value $r$ for which $U''(r) = 0$.

The investigation of a break at an inflection point poses some difficulties. The reason is that in the
intermediate regime that we are interested in, the chain remains in a position where
all particles are very nearly evenly spaced right up to the time when it breaks.
In such a situation of almost equal distances $s$, a Taylor expansion of $U'$ around $s$ shows that the effective
force which prevents each particle from leaving the position in the middle of its neighbours is equal to $U''(s)$.
A first problem is that this vanishes when $s$ approaches the inflection point $r$, and higher order terms of the expansion
take over and have to be dealt with. A second problem is that when some fluctuation eventually causes the distance between
two particles to exceed the critical distance $r$, another fluctuation may well
bring them back closer together before the deterministic part of the dynamics has had enough time to pull the chain apart.
It is therefore not even completely clear what the correct definition for a break time should be.

In \cite{allman_betz_11}, the first problem above was solved for the case when $d=3$ (one free particle) and $U'$ is a
suitable third order polynomial. Then, the random time when the size of one of the gaps first reaches the
inflection point can be analysed using the asymptotics of Airy functions. No extension
of this result to longer chains or to the second problem mentioned above exists. Since the problem has some similarities with
a metastability situation, methods from \cite{BovH15} might work,
but the time-dependence, the fact that all saddles of the total potential energy are of the same height to
leading order, and the detailed nature of the relevant questions (see below) mean that at least they would need to be extended
in a non-trivial way.


In the present work, we proceed as in \cite{allman_betz_09} and avoid both problems discussed above.
We assume that there is a fixed distance $\br > 1$ such that the chain breaks whenever two
neighbouring particles are $\br$ or more apart from each other. Physically, this can be justified by a sudden failure of the
molecular bonds beyond a certain threshold. Mathematically, we introduce the stopping times
\begin{equation} \label{eqn:def_tau_i}
  \tau^i_{\XX,\br}=\tau^i_{\XX,\br}(\eps,\sigma) := \inf\{ t\geq 0 ~|~ \XX^i_t - \XX^{i-1}_t = \br \}
\end{equation}
for $i\in\{1,\ldots,d\}$, and
\begin{equation} \label{eqn:def_tau}
  \tau_{\XX,\br} =\tau_{\XX,\br}(\eps,\sigma):= \min_{1\le	q i \leq d} \tau^i_{\XX,\br}
    = \inf\{ t\geq 0 : \exists i\in\{1,\ldots,d\} ~:~  \XX^i_t - \XX^{i-1}_t = \br \},
\end{equation}
and investigate their distributions. Furthermore, we assume that $U$ is strictly convex and increasing up to $\br$, thus guaranteeing
that a particle configuration with equal distances between particles is a stable equilibrium of the no-noise dynamics.
From the geometric point of view based on the observation of the process
$({\vecXX}_t)$, the break time simply means the exit time
of ${\vecXX}$ from a certain deterministic polytope.
Accordingly, we call $\tau^i_{\XX,\br}$, $\tau_{\XX,\br}$, and other similar variables {\it exit times}.


The case $d=3$ of a single free particle with strictly convex $U$ was
treated in \cite{allman_betz_09}. There, methods from \cite{BeGe06}
together with some symmetry considerations (possible due to $d=3$)
yielded the following dichotomy: if $\sigma \ll \eps$ (fast pulling regime),
the chain breaks deterministically at the right link,
while for $\sigma \gg \eps$, it breaks at each link with equal probability.
The threshold between both regimes was identified only
up to a logarithmic factor, and a large deviation (slow pulling) regime
where $1/\sigma$ is exponentially large in $1/\eps$ was not covered.
These shortcomings were overcome in the recent work \cite{ABL} for the
case where $U$ is quadratic. In that work, a detailed analysis of the
resulting Gaussian processes exhibited three regimes: the already
mentioned fast pulling regime, where the limiting quantities are governed
by the pulling force only; a slow pulling / large deviation regime where
the rupture is caused essentially solely by the noise; and an interesting
intermediate regime where both the pulling and the noise determine the
limiting characteristics.
In all of these regimes, \cite{ABL} provides precise asymptotics of the
break times and locations for chains of arbitrary length.

In the present paper, we extend the analysis beyond the case of quadratic potentials for the regime of intermediate pulling.
This regime is the most interesting one, since the fast pulling case is anyway dominated by the deterministic dynamics, while the breakage in  the slow
pulling case relies on a large deviation event and results will therefore depend on the details of the potential. In contrast,
for the intermediate pulling regime, we observe a universal behaviour for the break time distribution, in the sense that it only
depends on the curvature $U''(\br)$ of $U$ at the breaking distance.

The reason for this universality is not hard to understand intuitively.
Since the chain (of initial length $d$) must break once its total length exceeds
$d \br$, there is the simple, but important bound
\begin{equation} \label{eqn:t_ast}
   \tau_{\XX,\br} \leq t_\ast = t_\ast(\eps,\br) := d (\br-1) / \varepsilon.
\end{equation}
For the intermediate pulling regime, the break actually occurs when the chain is quite close to the maximal allowed elongation
$d \br$,
and is therefore initiated by
rather small fluctuations around the stable equilibrium. Therefore, the quadratic approximation to the potential is a good one
for this situation, and we obtain both the universality and the actual result by comparison with the relevant Gaussian processes.
Our proof follows this intuition.

Let us end this introduction by briefly discussing possible future extensions of our results. First of all, the assumption of strict
convexity of $U$ up to the break location is not harmless; in \cite{allman_betz_11}, where the situation with a break at an inflection
point of $U$ is studied, it is shown there that the
scaling of $\sigma$ with $\eps$ for the threshold between intermediate and fast pulling regime is different from the situation where we assume strict convexity. An extension of the results of \cite{allman_betz_11}
to longer chains in the spirit of the present paper would be very interesting,
but would need to do without the theory of Ornstein-Uhlenbeck processes that we use crucially. Another route for improvement would be
to investigate a chain of a length $d(\eps)$ that increases as $\eps \to 0$. Here, the obstacle to overcome is that in
\cite{ABL}, the spectral gap of a certain discrete Laplacian plays a crucial role, which disappears as the chain gets infinitely long.

\section{Main result}

In this section we give our main result and an outline of its proof.
We study the system \eqref{eqn:original_U} with the condition that the chain breaks when the distance between
a pair of neighbouring particles reaches the value $\br > 1$. The random time $\tau_{\XX,\br}^i$ at which
distance $\br$ is reached by the $i$-th link is thus given by \eqref{eqn:def_tau_i},
and the break time $\tau_{\XX,\br}$ of the chain by \eqref{eqn:def_tau}. Recall also the deterministic upper
bound $t_\ast$ on the break time given in \eqref{eqn:t_ast}.
We will make the following assumption for the potential $U$: \\[2mm]
\noindent {\bf Assumption $\U$.} The function $U$ is three times continuously differentiable
and $U''$ is strictly positive on $[1,\br]$.
\\[2mm]
We will investigate the
intermediate pulling regime characterized by the conditions
\begin{equation} \label{eqn:stretch_vb}
    \sigma/\eps\to \infty \qquad \text{and} \qquad \sigma^2 |\ln \eps|^3 \to 0.
\end{equation}
on the scaling parameters $\sigma$ and $\eps$. For stating our result, we define the quantities
\begin{equation}\label{eqn:parameterv0}
\begin{array}{l}
    v^2:=\frac{d-1}{2d},\qquad \gamma:=\sqrt{2} d v = \sqrt{d(d-1)},
\\[1mm]
   A_1:=A_d:=\frac{d}{d-1},\qquad A_i:=\frac{2d}{d-1}, i\in\{2,\ldots,d-1\},
\\[1mm]
 a_i:=vd A_i  / \sqrt{2\pi} \text{ for }i\in\{1,\ldots,d\}, \\[1mm]
a_0:=\sum_{i=1}^d a_i = 2vd^2 /\sqrt{2\pi}, \qquad  b:=\sqrt{2}/(vd).
\end{array}
\end{equation}
Recall that a random variable $\chi$ is {\it double exponential} (or Gumbel) with parameters $a,b>0$, if
$$
   \P( \chi \leq r ) = \exp( -a \exp( - b r)), \qquad r\in \R.
$$

\begin{thm}  \label{thm:nonlinear}
Let $(\XX_t^i)_{i=0, \ldots, d}$ solve the system \eqref{eqn:original_U},
where the potential $U$ satisfies Assumption $\U$. Let $\br > 1$, set $\uu := U''(\br)$,
and define $\tau_{\XX,\br}^i$, $\tau_{\XX,\br}$ and $t^\ast(\eps)$ as
in \eqref{eqn:def_tau_i}, \eqref{eqn:def_tau}
and \eqref{eqn:t_ast}, respectively. Then in the parameter regime described
by \eqref{eqn:stretch_vb}, we have
the following weak limit theorems for the break times as $\eps, \sigma \to 0$:
\begin{eqnarray} \label{eqn:intermediateexittimescaling_XX}
      \frac{\sqrt{\uu}\, \eps}{\sigma}\, \sqrt{\ln(\sigma/\eps)}
        \left( t_\ast(\eps,\br) - \gamma \frac{\sigma}{\sqrt{\uu} \eps} \sqrt{\ln (\sigma/\eps)}
        -\tau^i_{\XX,\br}(\eps,\sigma) \right)
     \tod \chi_i(\uu),
\\ \nonumber
    \qquad i\in \{1,\ldots,d\},
\end{eqnarray}
and
\begin{equation} \label{eqn:intermediateexittimescaling-tau_XX}
      \frac{\sqrt{\uu}\, \eps}{\sigma}\, \sqrt{\ln(\sigma/\eps)}
        \left( t_\ast(\eps,\br) - \gamma \frac{\sigma}{\sqrt{\uu} \eps} \sqrt{\ln (\sigma/\eps)}
        -\tau_{\XX,\br}(\eps,\sigma) \right)
     \tod \chi_0(\uu),
\end{equation}
where for each $i$, $\chi_i(\uu)$ is a double exponential random variable with parameters
$\sqrt{\uu}\, a_i, b$, respectively. \\
Moreover, under the same assumptions, we have
\be \label{eqn:positions}
   \P(\tau_{\XX,\br}=\tau_{\XX,\br}^i) \to \begin{cases}
                           \frac{1}{d-1} & i\in\{2,\ldots,d-1\};\\
                           \frac{1}{2(d-1)} & i\in\{1,d\}
                       \end{cases}
\ee
as $\sigma, \eps \to 0$.
\end{thm}

{\bf Remark.}  The above result is contained in \cite{ABL} for the special case $U(x)=x^2/2$; in that case, the processes
$\XX_t^i$ are Gaussian and can be analysed in great detail. Note that the break behaviour in the present result is indeed universal as it only depends
on the potential via $\uu$. From the relevant formulae it is also apparent that $\uu > 0$ is important for them to make sense;
this is another indication that in the case of break at an inflection point, as discussed in the introduction, a different
asymptotic behaviour should be expected.\\[2mm]

%

The first step in our proof consists in a slight generalization of the relevant result in \cite{ABL}. Put very succinctly,
it states that the statement of Theorem~\ref{thm:nonlinear} holds in the special case $U(x) = \uu x^2/2$ with $\uu > 0$.
Since we will need the notation in the proof later on anyway, we spell out the statement here.

Consider the linear system
\begin{equation} \label{eqn:original}
   \begin{cases}
     X_0^i = i & i=0,1,\ldots, d;\\
     X_t^0 = 0 & t\geq 0;\\
     X_t^d = d+\eps t & t\geq 0;\\
     \dd X_t^i = \uu(X^{i+1}_t+X^{i-1}_t -2 X^i_t)\dd t+\sigma \dd B_t^i,& i=1,\ldots,d-1, t\geq 0,
   \end{cases}
\end{equation}
where $\uu>0$ is a constant and we will abbreviate $\vecX_t:=(X^0_t,\ldots,X^d_t)^\top$.
As before, define the {\it break times} as
\[
  \tau^i=\tau^i(\eps,\sigma) := \inf\{ t\geq 0 ~|~ X^i_t - X^{i-1}_t = \br \}
\]
and
\[
  \tau =\tau(\eps,\sigma):= \min_{1\leq i \leq d} \tau^i
    = \inf\{ t\geq 0 : \exists i\in\{1,\ldots,d\} ~:~  X^i_t - X^{i-1}_t = \br \}.
\]
When we want to stress the dependence on the parameters $\uu$ and $\br$, we will write $\vecX_{\uu,t}$, $X^i_{\uu,t}$,
$\tau^i_{\uu,\br}(\eps,\sigma)$ and $\tau_{\uu,\br}(\eps,\sigma)$.

We consider scaling regime where
\begin{equation} \label{eqn:stretch}
    \sigma/\eps\to \infty \qquad \text{and} \qquad \sigma^2 |\ln \eps| \to 0.
\end{equation}
Note that this regime is slightly wider than the one given in \eqref{eqn:stretch_vb}; the small difference
is due to the fact that close to the large deviation regime, the Gaussian and the non-Gaussian processes start to look differently.
We  have

\begin{thm} \label{thm:intermediateslowpulling_ub}
Assume that $\eqref{eqn:stretch}$ holds. Then, as $\eps,\sigma\to 0$,
the analogue of \eqref{eqn:positions} holds for $\tau_{\uu,\br}(\eps,\sigma)$ and $\tau^i_{\uu,\br}(\eps,\sigma)$
and we have the following weak limit theorems for the break times:
\begin{eqnarray} \label{eqn:intermediateexittimescaling_ub}
     \frac{\sqrt{\uu}\, \eps}{\sigma}\, \sqrt{\ln(\sigma/\eps)}
        \left( t_\ast(\eps,\br) - \gamma \frac{\sigma}{\sqrt{\uu} \eps} \sqrt{\ln (\sigma/\eps)}
        -\tau^i_{\uu,\br}(\eps,\sigma) \right)
     \tod \chi_i(\uu),
\\ \nonumber
     \qquad i\in \{1,\ldots,d\},
\end{eqnarray}
and
\begin{equation} \label{eqn:intermediateexittimescaling-tau_ub}
      \frac{\sqrt{\uu}\, \eps}{\sigma}\, \sqrt{\ln(\sigma/\eps)}
        \left( t_\ast(\eps,\br) - \gamma \frac{\sigma}{\sqrt{\uu} \eps} \sqrt{\ln (\sigma/\eps)}
        -\tau_{\uu,\br}(\eps,\sigma) \right)
     \tod \chi_0(\uu),
\end{equation}
where $\chi_i(\uu)$ is a double exponential random variable with parameters
$\sqrt{\uu} a_i, b$, and the values $a_i, b$ are defined in \eqref{eqn:parameterv0}.
\end{thm}


\cite{ABL} contains Theorem~\ref{thm:intermediateslowpulling_ub} for the special case $\br = 2$, $\uu = 1$.
The passage to general $\br$ and $\uu$ is made by standard scaling arguments, which we spell out in Section
\ref{sec:scaling} for the convenience of the reader.

The second step is to replace $\uu$ in the system \eqref{eqn:original} by a time-dependent quantity $\phi(t)$.
Again, the short version of the result is that the statement of Theorem~\ref{thm:intermediateslowpulling_ub} remains true
with $\uu$ replaced by $\phi(t)$, but
we spell out the result as we will need the notation later in the proofs anyway.

Let $\vecZ_t=(Z_t^0,\ldots,Z^{d}_t)^\top$
solve the linear system with time-dependent coefficient
\begin{equation} \label{eqn:original_phi}
   \begin{cases}
     Z_0^i = i & i=0,1,\ldots, d;\\
     Z_t^0 = 0 & t\geq 0;\\
     Z_t^d = d+\eps t & t\geq 0;\\
     \dd Z_t^i = \phi(t)(Z^{i+1}_t+Z^{i-1}_t -2 Z^i_t)\dd t+\sigma \dd B_t^i,& i=1,\ldots,d-1, t\geq 0,
   \end{cases}
\end{equation}

Here we immediately take $\phi(t):=U''(q_t)$ with
$$
   q_t := 1 + \eps \frac{t}{d},
$$
where $U$ fulfills Assumption $\U$. This is the correct
linearization of \eqref{eqn:original_U} in the following sense: by the assumptions on $U$, the potential energy
$\boldsymbol U(x) = \sum_{i=1}^d U(x_{i} - x_{i-1})$ is minimized by the vector $(i (1+ \eps \frac{t}{d}))_{0 \leq i \leq d}$
at time $t$. Since the pulling is slow and the noise is small, the system will be close to that energy minimum at all times,
and as terms of order zero and one cancel when we Taylor expand each $U(x_i - x_{i-1})$ around $1+ \eps \frac{t}{d}$, so that $\phi(t)=U''(1+\eps \frac{t}{d})$
is the dominant term.
Notice in particular that for quadratic potentials $U(x)=\uu\, x^2/2$, we have $\phi(\cdot)\equiv\uu$ and
$(\vecZ_t)$ coincides with $(\vecX_{\uu,t})$.

Since $\vecZ = (Z^0, \ldots, Z^d)^\top$ is still a Gaussian process, we will be able to analyse it in great detail.
We work in the scaling regime
\begin{equation} \label{eqn:stretch_Z}
    \sigma/\eps\to \infty \qquad \text{and} \qquad \sigma^2 |\ln \eps|^{3/2} \to 0,
\end{equation}
which is larger than the one for Theorem~\ref{thm:nonlinear} but smaller than the one for
Theorem~\ref{thm:intermediateslowpulling_ub} due to our need to accommodate the fact that the coefficient of the linear
force now depends on time. As above, we introduce the break times
\[
  \tau^i_{Z,\br}=\tau^i_{Z,\br}(\eps,\sigma) := \inf\{ t\geq 0 ~|~ Z^i_t - Z^{i-1}_t = \br \}
\]
and
\[
  \tau_{Z,\br} =\tau_{Z,\br}(\eps,\sigma):= \min_{1\leq i \leq d} \tau^i_{Z,\br}
    = \inf\{ t\geq 0 : \exists i\in\{1,\ldots,d\} ~:~  Z^i_t - Z^{i-1}_t = \br \}.
\]

We have

\begin{thm} \label{thm:intermediateslowpulling_Z}
Assume that $\eqref{eqn:stretch_Z}$ holds. Fix $\br>1$ and set $\uu:=U''(\br)$. Then, as $\eps,\sigma\to 0$,
the analogue of \eqref{eqn:positions} holds for $\tau_{Z,\br}(\eps,\sigma)$ and $\tau^i_{Z,\br}(\eps,\sigma)$
and we have the following weak limit theorems for the break times:
\begin{eqnarray} \label{eqn:intermediateexittimescaling_Z}
      \frac{\sqrt{\uu}\, \eps}{\sigma}\, \sqrt{\ln(\sigma/\eps)}
        \left( t_\ast(\eps,\br) - \gamma \frac{\sigma}{\sqrt{\uu} \eps} \sqrt{\ln (\sigma/\eps)}
        -\tau^i_{Z,\br}(\eps,\sigma) \right)
     \tod \chi_i(\uu),
\\ \nonumber
    \qquad i\in \{1,\ldots,d\},
\end{eqnarray}
and
\begin{equation} \label{eqn:intermediateexittimescaling-tau_Z}
      \frac{\sqrt{\uu}\, \eps}{\sigma}\, \sqrt{\ln(\sigma/\eps)}
        \left( t_\ast(\eps,\br) - \gamma \frac{\sigma}{\sqrt{\uu} \eps} \sqrt{\ln (\sigma/\eps)}
        -\tau_{Z,\br}(\eps,\sigma) \right)
     \tod \chi_0(\uu),
\end{equation}
where $\chi_i(\uu)$ is a double exponential random variable with parameters
$\sqrt{\uu}\, a_i, b$, and $a_i,b$ are defined in \eqref{eqn:parameterv0}.
\end{thm}

The rest of this paper is structured as follows. As already mentioned, Theorem~\ref{thm:intermediateslowpulling_ub} is proved in Section~\ref{sec:scaling}, while Theorem~\ref{thm:intermediateslowpulling_Z} is proved in Section~\ref{sec:XtoZ}. It is used
to prove Theorem~\ref{thm:nonlinear} in Section~\ref{sec:ZtoXX}. Proofs of various technical results used in these proofs are
collected in Section~\ref{sec:technical}.

\bigskip

\section{Scaling of the standard model}
\label{sec:scaling}

Here we prove Theorem~\ref{thm:intermediateslowpulling_ub}. Recall that the case where $\uu = 1$, $\br = 2$ has been done in
\cite{ABL}. Below, we will refer to this case as the {\em standard problem}.

{\sl Step 1. Scaling of $\uu$.}
We first consider the scaling of $\uu$ with fixed $\br$. Let $\tvecX$ be a solution
of the system \eqref{eqn:original} with parameters $\uu=1$, $\teps$ and $\tsigma$. We make the time change
\[
   \hX^i_t:= \tX^i_{\uu t}.
\]
This family solves the system with parameters $\uu$, $\eps:=\uu\, \teps$, and
$\sigma:=\sqrt{\uu}\, \tsigma$, because $\hX^0_t=\tX^0_{\uu\,t}=0$,
$\hX^d_t=\tX^d_{\uu\, t}= d+ \teps(\uu\, t)=  d+ (\uu\, \teps) t$, and an easy
calculation shows that
\[
 \dd \hX_t^i = \uu(\hX^{i+1}_t+\hX^{i-1}_t -2 \hX^i_t)\dd t
 +\sqrt{\uu}\, \tsigma \dd \hB_t^i,\qquad i=1,\ldots,d-1, t\geq 0,
\]
with appropriate independent Brownian motions $\hB_t^i$.

The above argument shows that for given $\eps,\sigma$,
we have
\[
  \vecX_{\uu,t} (\eps,\sigma) \deq \vecX_{1, \uu\,t} \left(\frac{\eps}{\uu},\frac{\sigma}{\sqrt{\uu}}\right),
\]
the latter being a solution of the standard problem,
yet with modified parameters.

Next, having a relation between the processes, we obtain the relations between the
break times for arbitrary $\br>1$,
\begin{eqnarray} \label{eqn:ub1b}
   \tau^i_{\uu,\br}(\eps,\sigma) &=&  \uu^{-1}\ \tau^i_{1,\br}(\eps/\uu,\sigma/\sqrt{\uu}),
\\
   \tau_{\uu,\br}(\eps,\sigma) &=&  \uu^{-1}\  \tau_{1,\br}(\eps/\uu,\sigma/\sqrt{\uu}).
\end{eqnarray}
 \medskip

{\sl Step 2. Scaling of $\br$.} Now we fix $\uu=1$ and scale $\br$.
Let us fix a break position $i\in \{1,\ldots,d\}$. The break condition is
\[
   X^i_t-X^{i-1}_t =\br.
\]

This is equivalent to
\[
   \sigma V_t^i + \eps \, \frac{t}{d} +\eps\, R^i_t =\br-1 ,
\]
where $V^i$ is an appropriate  asymptotically stationary process, and $R^i$ is a bounded deterministic function,
both defined in \cite{ABL}. Dividing by $\br-1$ yields the equivalent form
\[
   \frac{\sigma}{\br-1}\, V_t^i + \frac{\eps}{\br-1} \, \frac{t}{d} +  \frac{\eps}{\br-1}\, R^i_t = 1,
\]
which coincides with the break condition for the standard case $\br=2$ with parameters
$\tfrac{\eps}{1-\br}, \tfrac{\sigma}{1-\br}$. We conclude that $\tau^i_{1,\br}(\eps,\sigma)$
coincides in distribution with  $\tau^i_{1,2}\left(\tfrac{\eps}{\br-1}, \tfrac{\sigma}{\br-1}\right)$.

{\sl Step 3. Combining two scalings.}
By combining the results of two scalings we see that the vector
 $\tau^i_{\uu,\br}(\eps,\sigma)$, $i=1,\ldots,d$,
has the same distribution as the vector
$\uu^{-1} \tau^i_{1,2}\left(\tfrac{\eps}{\uu(\br-1)}, \tfrac{\sigma}{\sqrt{\uu}(\br-1)})\right)$, $i=1,\ldots,d$.
Subsequently, this also holds for $\tau_{\uu,\br}=\min_i \tau^i_{\uu,\br}$ and $\tau_{1,2}=\min_i \tau_{1,2}^i$.

{\sl Step 4. Weak convergence.}
From \cite{ABL} we know that Theorem~\ref{thm:intermediateslowpulling_ub} is valid in the standard problem.
We apply this version of the theorem to the modified pair of parameters
$\left(\tfrac{\eps}{\uu(\br-1)}, \tfrac{\sigma}{\sqrt{\uu}(\br-1)})\right)$
and obtain
\[
       \frac{\eps}{\sigma\sqrt{\uu}}\, \sqrt{\ln\left( \frac{\sigma\sqrt{\uu}}{\eps}\right)}
        \left(\frac{d\uu(\br-1)}{\eps}
                - \gamma \frac{\sigma\sqrt{\uu}}{\eps}
                \sqrt{\ln\left( \frac{\sigma\sqrt{\uu}}{\eps}\right)}
           -\uu \,\tau^i_{\uu,\br}\left(\eps,\sigma\right)
   \right) \tod \chi_i.
\]
or, equivalently,
\[
       \frac{\eps\sqrt{\uu}}{\sigma}\, \sqrt{\ln\left( \frac{\sigma\sqrt{\uu}}{\eps}\right)}
        \left(\frac{d(\br-1)}{\eps}
                - \gamma \frac{\sigma}{\eps\sqrt{\uu}}
                \sqrt{\ln\left( \frac{\sigma\sqrt{\uu}}{\eps}\right)}
           -\tau^i_{\uu,\br}\left(\eps,\sigma\right)
   \right) \tod \chi_i.
\]

Next, using expansion
\begin{eqnarray*}
   \sqrt{\ln\left(\frac{\sigma\sqrt{\uu}}{\eps}\right)}
  &=&  \sqrt{\ln\left(\frac{\sigma}{\eps}\right) +\frac{\ln \uu}{2}}
   =  \sqrt{\ln\left(\frac{\sigma}{\eps}\right)}\sqrt{ 1+ \frac{\ln \uu}{2\ln\left(\frac{\sigma}{\eps}\right) }}
\\
  &=&  \sqrt{\ln\left(\frac{\sigma}{\eps}\right)} \ \left(1+ \frac{\ln \uu}{4\ln\left(\frac{\sigma}{\eps}\right) }(1+o(1)) \right)
\\
   &=&  \sqrt{\ln\left(\frac{\sigma}{\eps}\right)} + \frac{\ln \uu}{4\sqrt{\ln\left(\frac{\sigma}{\eps}\right)} }(1+o(1))
\end{eqnarray*}
 we obtain
\[
      \frac{\eps\sqrt{\uu}}{\sigma}\, \sqrt{\ln\left( \frac{\sigma}{\eps}\right)}
        \left(\frac{d(\br-1)}{\eps}
                - \gamma \frac{\sigma}{\eps\sqrt{\uu}}
                \sqrt{\ln\left( \frac{\sigma}{\eps}\right)}
           -\tau^i_{\uu,\br}\left(\eps,\sigma\right)
   \right)
     \tod \chi_i + \frac{\gamma \ln \uu}{4}.
\]

Notice that a shifted variable having a double exponential distribution belongs to the same class:
If $\kappa\in\R$ and $\xi$ is double exponential with parameters $a,b$, then $\xi+\kappa$ is also double exponential
with parameters $a\exp\{b\kappa\},b$.

In our case, $\kappa= \tfrac{\gamma\ln \uu }{4}$, hence $\exp\{b\kappa\}= \uu^{b\gamma/4}$.
Furthermore, an easy calculation shows that $b\gamma=\sqrt{2} (vd)^{-1} \cdot \sqrt{2}dv=2$.
Therefore, $\exp\{b\kappa\}= \sqrt{\uu}$. We finally obtain
\begin{equation} \label{eqn:iets6}
         \frac{\eps\sqrt{\uu}}{\sigma}\, \sqrt{\ln\left( \frac{\sigma}{\eps}\right)}
        \left(\frac{d(\br-1)}{\eps}
                - \gamma \frac{\sigma}{\eps\sqrt{\uu}}
                \sqrt{\ln\left( \frac{\sigma}{\eps}\right)}
           -\tau^i_{\uu,\br}\left(\eps,\sigma\right)
   \right)
     \tod \chi_i(\uu),
\end{equation}
where $\chi_i(\uu)$ is double exponential with parameters $\sqrt{\uu}a_i,b$.

The latter equation coincides with the claim \eqref{eqn:intermediateexittimescaling_ub}.
The proof of \eqref{eqn:intermediateexittimescaling-tau_ub} is exactly the same.
Relation  \eqref{eqn:positions} follows as a by-product of the scaling reductions.
Therefore, the proof of Theorem~\ref{thm:intermediateslowpulling_ub} is finished.

\begin{rem} \label{rem:uniform} {\rm We see from the proof that the weak convergence
in \eqref{eqn:intermediateexittimescaling_ub}
and \eqref{eqn:intermediateexittimescaling-tau_ub} is locally uniform in $\br$
(as long as $\br$ is bounded away from $1$) for every fixed $\uu$.
We will use this fact later on.}
\end{rem}

\section{Comparing the Gaussian processes $\vecX_{\uu,\br}$ and $\vecZ$}
\label{sec:XtoZ}

Before we start with the proof of Theorem~\ref{thm:intermediateslowpulling_Z},
we need two auxiliary results. The first one shows that $\vecZ$ has no early breaks.
Recall the notation $t_\ast= t_\ast(\eps,\br) := d (\br-1) / \varepsilon$.

\begin{lem}\label{lem:noearlyZ} There exists $\gamma_1=\gamma_1(d,U(\cdot))$ such that
\begin{equation} \label{eqn:noearlyZ}
  \P\left(  \tau_{Z,\br}(\eps,\sigma) < t_\ast
     - \gamma_1 \, \frac{\sigma}{\eps} \sqrt{\ln\left(\frac{\sigma}{\eps}\right)}
   \right) \to 0, \qquad \textrm{ as } \eps,\sigma\to 0.
\end{equation}
\end{lem}

From Theorem~\ref{thm:intermediateslowpulling_ub} we know that the same statement
is true for $\vecX_\uu$, namely, for any $\gamma_1>\gamma/\sqrt{\uu}$,
\begin{equation} \label{eqn:noearlyXu}
  \P\left(  \tau_{\uu,\br}(\eps,\sigma) < t_\ast
     - \gamma_1 \, \frac{\sigma}{\eps} \sqrt{\ln\left(\frac{\sigma}{\eps}\right)}
   \right) \to 0.
\end{equation}
But we also need \eqref{eqn:noearlyZ} for $\vecZ$.
The proof of Lemma~\ref{lem:noearlyZ} is given in Section~\ref{sec:noearlyZ}.

The second auxiliary lemma shows that $\vecZ$ and $\vecX_\uu$ are close on an important time interval.
\begin{lem} \label{lem:ZXclose}
Let $\gamma_1>0$  and $\theta>0$. Assume \eqref{eqn:stretch_Z}.
Denote $\delta=\delta(\eps,\sigma):= \frac{\theta\,\sigma}{\sqrt{\ln\left(\frac{\sigma}{\eps}\right)}}$.
Then,  as $\eps,\sigma\to 0$, for every $i=0,1,..,d$, we have
\begin{equation} \label{eqn:ZXclose}
    \P\left(  \sup_{t\in [t_\ast - \gamma_1 \, \frac{\sigma}{\eps} \sqrt{\ln\left(\frac{\sigma}{\eps}\right)}, t_\ast]}
   |Z^i_t-X^i_{\uu,t}|>\delta \right) \to 0.
\end{equation}
\end{lem}

The proof of Lemma~\ref{lem:ZXclose} is given in Section~\ref{sec:ZXclose}.
\bigskip

\begin{proof}[ of Theorem~\ref{thm:intermediateslowpulling_Z}]
Let $\gamma_1$  be chosen large enough to satisfy Lemma~\ref{lem:noearlyZ}.
Let
$t\in [t_\ast - \gamma_1 \, \frac{\sigma}{\eps} \sqrt{\ln\left(\frac{\sigma}{\eps}\right)},t_\ast]$ be such that
\[
   \max_{1\leq i\leq d}  \left(Z^{i}_t-Z^{i-1}_t\right) \geq \br.
\]
If, additionally,
\begin{equation} \label{eqn:ZXdelta}
  \max_{0\leq i\leq d} |Z^i_t-X^i_{\uu,t}|<\delta,
\end{equation}
 then
\[
   \max_{1\leq i\leq d} \left( X^{i}_{\uu,t}-X^{i-1}_{\uu,t}\right) \geq \br-2\delta.
\]
It follows that
$\tau_{\uu,\br-2\delta} \leq t$. By using Lemma~\ref{lem:noearlyZ} and Lemma~\ref{lem:ZXclose} we obtain
\[
  \P\left( \tau_{Z,\br}\leq t \right) \leq \P\left( \tau_{\uu,\br-2\delta}\leq t \right) +o(1)
\]
with $o(1)$ uniformly over
$t\in [t_\ast - \gamma_1 \, \frac{\sigma}{\eps} \sqrt{\ln\left(\frac{\sigma}{\eps}\right)}, t_\ast]$.

Let $r\in \R$ be fixed. The previous equation and the definition of $\delta$ yield
\begin{eqnarray*}
   P_r &:=&
   \P\left(
     \frac{\sqrt{\uu}\, \eps}{\sigma}\, \sqrt{\ln(\sigma/\eps)}
        \left( t_\ast - \gamma \frac{\sigma}{\sqrt{\uu} \eps} \sqrt{\ln (\sigma/\eps)}
        -\tau_{Z,\br}(\eps,\sigma) \right) \geq r
   \right)
\\
  &=&    \P\left( \tau_{Z,\br}(\eps,\sigma) \leq t_\ast
       -  \gamma \frac{\sigma}{\sqrt{\uu} \eps} \sqrt{\ln (\sigma/\eps)}
       - \frac{r\sigma} {\sqrt{\uu}\, \eps \sqrt{\ln(\sigma/\eps)}}
         \right)
\\
    &\leq&    \P\left( \tau_{\uu,\br-2\delta}(\eps,\sigma) \leq t_\ast
       -  \gamma \frac{\sigma}{\sqrt{\uu} \eps} \sqrt{\ln (\sigma/\eps)}
       - \frac{r\sigma} {\sqrt{\uu}\, \eps \sqrt{\ln(\sigma/\eps)}}
         \right) + o(1)
\\
    &=&    \P\left( \tau_{\uu,\br-2\delta}(\eps,\sigma) \leq
        \frac{d(\br-2\delta-1)}{\eps} +  \frac{2d\delta}{\eps}
       -  \gamma \frac{\sigma}{\sqrt{\uu} \eps} \sqrt{\ln (\sigma/\eps)}
       - \frac{r\sigma} {\sqrt{\uu}\, \eps \sqrt{\ln(\sigma/\eps)}}
         \right) + o(1),
\\
    &=&    \P\left( \tau_{\uu,\br-2\delta}(\eps,\sigma) \leq
        t_\ast(\eps,\br-2\delta)
       -  \gamma \frac{\sigma}{\sqrt{\uu} \eps} \sqrt{\ln (\sigma/\eps)}
       - \frac{(r+2d\theta \sqrt{\uu})\sigma} {\sqrt{\uu}\, \eps \sqrt{\ln(\sigma/\eps)}}
         \right) + o(1),
\\
      &=& \P\left( \frac{\sqrt{\uu}\, \eps}{\sigma}\, \sqrt{\ln(\sigma/\eps)}
        \left( t_\ast(,\eps,\br -2\delta) - \gamma \frac{\sigma}{\sqrt{\uu} \eps} \sqrt{\ln (\sigma/\eps)}
        -\tau_{\uu,\br-2\delta}(\eps,\sigma) \right) \geq r+2d\theta \sqrt{\uu}
         \right) + o(1),
\end{eqnarray*}
By applying Theorem~\ref{thm:intermediateslowpulling_ub} with $\widetilde{\br}:= \br-2\delta$ instead of $\br$, we
obtain\footnote{We stress that $\widetilde{\br}$ slightly depends of $\eps,\sigma$ through $\delta$. Therefore, we need a uniform
version of the theorem, cf. Remark~\ref{rem:uniform}.}
\[
  \limsup_{\eps,\sigma\to 0} P_r \leq \P(\chi_0(\uu) \geq r+2d\theta \sqrt{\uu}).
\]
Finally, by letting $\theta\to 0$ we obtain the desired upper bound
\[
  \limsup_{\eps,\sigma\to 0} P_r \leq \P(\chi_0(\uu) \geq r).
\]
The lower bound follows by the same lines.
If for some $t\in [t_\ast - \gamma_1 \, \frac{\sigma}{\eps} \sqrt{\ln\left(\frac{\sigma}{\eps}\right)}, t_\ast]$
we have
\[
   \max_{1\leq i\leq d} \left( X^{i}_{\uu,t}-X^{i-1}_{\uu,t}\right) \geq \br + 2\delta
\]
and \eqref{eqn:ZXdelta} holds, then
\[
   \max_{1\leq i\leq d}  \left(Z^{i}_t-Z^{i-1}_t\right) \geq \br.
\]
In other words, $\tau_{Z,\br} \leq t$.

By using \eqref{eqn:noearlyXu} and Lemma~\ref{lem:ZXclose} we obtain
\[
  \P\left( \tau_{Z,\br}\leq t \right) \geq \P\left( \tau_{\uu,\br+2\delta}\leq t \right) -o(1),
\]
and the rest of the derivation leading to
\[
  \liminf_{\eps,\sigma\to 0} P_r \geq \P(\chi_0(\uu) \geq r).
\]
continues as above, $\br-2\delta$ being replaced with $\br+2\delta$.
\end{proof}

\section{Break times for the non-linear system}
\label{sec:ZtoXX}

In order to obtain information about the process $\vecXX$ from the analysis of $\vecZ$,
we must investigate the difference between  $\vecZ$ and $\vecXX$ .
We first need some definitions and notations, assuming that $\vecZ$ and $\vecXX$ are defined
on the same probability space, and are driven by the same Brownian motions.
On this probability space we define the stochastic processes
\[
   S_t^\ast :=  \sup_{0 \leq s \leq t} \| \vecZ_s - \vecXX_s\|_2
   = \sup_{0 \leq s \leq t} \Big( \sum_{i=0}^d | Z_s^i - \XX_s^{i}|^2 \Big)^{1/2}
   = \sup_{0 \leq s \leq t} \Big( \sum_{i=1}^{d-1} | Z_s^i - \XX_s^{i}|^2 \Big)^{1/2}
\]
and
\[
   M^\ast_t := \sup_{0 \leq s \leq t} \Big( \sum_{i=1}^d | Z_s^i - Z_s^{i-1} - q_s|^2 \Big)^{1/2},
\]
where we recall that $q_s= 1+\tfrac{\eps s}{d}$.

The key approximation result is as follows.

\begin{prop} \label{prop:vbgeneral}
Let $r$ be as in Assumption $\U$. Then  there exists a large constant $C$ depending on $d$ and
on the potential $U$  such that for all  $t \leq t_\ast$ and all $\delta>0$ we have
\begin{equation} \label{eqn:vbgeneral}
   \P(S_t^*\geq \delta)
   \leq \P(M_t^*\geq \sqrt{\delta}/C)
        +\P(M_t^*+4C(M_t^*)^2\geq r)+\P(C M_t^*\geq 1).
\end{equation}
\end{prop}

The advantage of this proposition is that it evaluates the difference
of the two processes in terms of the Gaussian process alone.
Its proof is given in Section~\ref{sec:proof_prop:vbgeneral}.

Based on this general estimate, we obtain a specific bound suited for the theorem's proof.

\begin{prop} \label{prop:vbspecial}
Assume that \eqref{eqn:stretch_vb} holds.
Let $\theta>0$ and
$\delta=\delta(\eps,\sigma):= \frac{\theta\,\sigma}{\sqrt{\ln\left(\frac{\sigma}{\eps}\right)}}$.
 Then
\begin{equation}
   \lim_{\eps,\sigma\to 0} \P(S_{t_\ast}^*\geq \delta)=0.
\end{equation}
\end{prop}

The proof of this proposition is given in Section~\ref{sec:proof_prop:vbspecial}.

\begin{proof} [ of Theorem~\ref{thm:nonlinear}]

{\sl Lower bound.} Assume that for some $s$ we have
\begin{eqnarray*}
  &&   \max_{1\leq i\leq d} \left( X^{i}_{\uu,s}-X^{i-1}_{\uu,s}\right) \geq \br+ 4\delta,
\\
  &&  \max_{0\leq i\leq d} |X^i_{\uu,s}-Z^i_s|<\delta,
\\
  &&  \max_{0\leq i\leq d} |Z^i_s-\XX^i_{\uu,s}|<\delta.
\end{eqnarray*}
Then
\[
   \max_{1\leq i\leq d} \left( \XX^{i}_{\uu,s}-\XX^{i-1}_{\uu,s}\right) \geq \br.
\]
It follows that for every $\gamma_1>\gamma/\sqrt{\uu}$ and  every $t\in [t_\ast - \gamma_1 \, \frac{\sigma}{\eps} \sqrt{\ln\left(\frac{\sigma}{\eps}\right)}, t_\ast]$
\begin{eqnarray*}
    \P(\tau_{\XX,\br} \leq t) &\geq&  \P(\tau_{\uu,\br+4\delta} \leq t)
    -\P\left(\tau_{\uu,\br} \leq t_\ast - \gamma_1 \, \frac{\sigma}{\eps} \sqrt{\ln\left(\frac{\sigma}{\eps}\right)}
    \right)
\\
    && -  \sum_{i=1}^{d-1} \P\left(  \sup_{s\in [t_\ast - \gamma_1 \, \frac{\sigma}{\eps} \sqrt{\ln\left(\frac{\sigma}{\eps}\right)},t_\ast]}
   |X^i_{\uu,s}-Z_s^i|>\delta \right)
\\
   && -  \P(S_{t_\ast}^*\geq \delta)
\\
  &=&  \P(\tau_{\uu,\br+4\delta} \leq t) -o(1),
\end{eqnarray*}
where we used Theorem~\ref{thm:intermediateslowpulling_ub},
Lemma~\ref{lem:ZXclose} and Proposition~\ref{prop:vbspecial}.
The rest of the proof goes along the same lines as in Theorem~\ref{thm:intermediateslowpulling_Z}.

{\sl Upper bound.} Assume that for some $s$ we have
\begin{eqnarray*}
  &&   \max_{1\leq i\leq d} \left( \XX^{i}_{\uu,s}-\XX^{i-1}_{\uu,s}\right) \geq \br,
\\
  &&  \max_{0\leq i\leq d} |Z^i_s-\XX^i_{\uu,s}|<\delta,
\end{eqnarray*}
Then
\[
  \max_{1\leq i\leq d} \left( Z^{i}_{s}- Z^{i-1}_{s}\right) \geq \br-2\delta.
\]
If, additionally
\[
    \max_{0\leq i\leq d} |X^i_{\uu,s}-Z^i_s|<\delta,
\]
then
\[
   \max_{1\leq i\leq d} \left( X^{i}_{\uu,s}- X^{i-1}_{\uu,s}\right) \geq \br-4\delta.
\]
It follows that for every  $t\in [t_\ast- \gamma_1 \, \frac{\sigma}{\eps} \sqrt{\ln\left(\frac{\sigma}{\eps}\right)},t_\ast]$
\begin{eqnarray*}
    \P(\tau_{\XX,\br} \leq t) &=&
    \P\left(\tau_{\XX,\br} \leq t_\ast - \gamma_1 \, \frac{\sigma}{\eps} \sqrt{\ln\left(\frac{\sigma}{\eps}\right)}\right)
    + \P\left(\tau_{\XX,\br} \in [t_\ast - \gamma_1 \, \frac{\sigma}{\eps} \sqrt{\ln\left(\frac{\sigma}{\eps}\right)},t]\right)
\\
    &\leq&
     \P\left( \tau_{Z,\br-2\delta} \leq t_\ast - \gamma_1 \, \frac{\sigma}{\eps} \sqrt{\ln\left(\frac{\sigma}{\eps}\right)}\right)
    +  \P(\tau_{\uu,\br-4\delta} \leq t)
\\
    && +  \sum_{i=1}^{d-1} \P\left( \sup_{s\in [t_\ast - \gamma_1 \, \frac{\sigma}{\eps} \sqrt{\ln\left(\frac{\sigma}{\eps}\right)},t_\ast]}
   |X^i_{\uu,t}-Z_t^i|>\delta \right) +  \P(S_{t_\ast}^*\geq \delta)
\\
  &=&  \P(\tau_{\uu,\br-4\delta} \leq t) +o(1),
\end{eqnarray*}
where we used Lemma~\ref{lem:noearlyZ}, Lemma~\ref{lem:ZXclose} and Proposition \ref{prop:vbspecial}.
The rest of the proof goes along the same lines as in Theorem~\ref{thm:intermediateslowpulling_Z}.
\end{proof}

\section{Auxiliary technical results}
\label{sec:technical}

\subsection{Some properties of the process $\vecZ$}
\subsubsection*{Scalar analogues of $\vecZ$}

We work under the assumptions \eqref{eqn:stretch_Z} and $\U$.
Recall that under assumption $\U$  our potential $U$ is convex on $[1,\br]$ and there exist finite positive constants
$\kappa_{\min}, \kappa_{\max}$ and $K$ such that
\begin{eqnarray*}
   \kappa _{\min}  \leq  U''(x) &\le&  \kappa_{\max}, \qquad x\in [1,\br],
\\
  |U'''(x)| &\le& K, \ \ \qquad x\in [1,\br].
\end{eqnarray*}
\smallskip

\begin{lem} \label{lem:winternoteslemma}
Let $U : [1,\br] \to \R$ be a function satisfying Assumption $\U$.
Set $\phi(t):=U''(1+\frac{\eps t}{d})$, $t\in[0,t_\ast]$
and $\uu:=\phi(t_\ast)=U''(\br)$, where $t_\ast:=d(\br-1)/\eps$.

Consider two scalar stochastic differential equations
\[
     \dd \newY_t = -\uu\, \newY_t\dd t+\sigma \dd B_t,\qquad t\geq 0,\qquad \newY_0=0,
\]
and
\begin{equation} \label{eqn:newZdefn}
     \dd \ZZZ_t = -\phi(t)\ZZZ_t\dd t+\sigma \dd B_t,\qquad t\geq 0,\qquad \ZZZ_0=0.
\end{equation}
Assume that $\eps,\sigma\to 0$ with \eqref{eqn:stretch_Z} being true.
Let
\[
   \TT=\TT(\sigma,\eps,\gamma_1):= [t_\ast- (\gamma_1 \sigma/\eps) \sqrt{\ln (\sigma/\eps)},t_\ast].
\]
Then, for any $\gamma_1>0$ and $\theta>0$, we have
\[
     \P\left( \sup_{t\in \TT} |\newY_t-\ZZZ_t|
     > \frac{\theta\sigma}{\sqrt{\ln (\sigma/\eps)}} \right) \to 0.
\]
\end{lem}

\begin{proof} We will use the abbreviations $h:=\eps/\sigma\to 0$ and $\psi:=\ln (h^{-1}) \to \infty$.

{\it Step 1:} We derive formulas for the variances of $\newY$, $\ZZZ$ and their covariances.
Note that the explicit solution of \eqref{eqn:newZdefn} is given by
\begin{equation} \label{eqn:defnnewz}
       \ZZZ_t =  \sigma\, \exp(-\Phi(t)) \int_0^t \exp( \Phi(s) )\, \dd B_s,
\end{equation}
where $\Phi(t):=\int_0^t \phi(u) \dd u$.
Therefore, for the covariance we have
$$
     \cov(\ZZZ_{t_1},\ZZZ_{t_2})
     = \sigma^2 \int_0^{\min(t_1,t_2)} \exp( 2\Phi(s)-\Phi(t_1)-\Phi(t_2) ) \dd s.
$$
As a special case, for the variance we obtain
\begin{equation} \label{eqn:newZvariance}
    \var \ZZZ_t  = \sigma^2 \int_0^t \exp( 2(\Phi(s)-\Phi(t)) ) \dd s.
\end{equation}

Replacing $\phi(t)$ by the constant $\uu$ in these formulas yields
\begin{eqnarray}
\label{eqn:addarg210+}
     \newY_t &=& \sigma \, \int_0^t \exp( \uu(s-t) ) \, \dd B_s\,;
\\   \nonumber
     \cov(\newY_{t_1},\newY_{t_2})
     &=& \sigma^2 \int_0^{\min(t_1,t_2)} e^{ \uu(2s-t_1-t_2) } \dd s
\\ \label{eqn:newYcovariance}
     &=& \frac{\sigma^2}{2\uu} \left( e^{- \uu|t_1-t_2|}-e^{-\uu(t_1+t_2)}\right);
\\  \label{eqn:newYvariance}
   \var\newY_t  &=& \frac{\sigma^2}{2\uu}\left( 1 - e^{-2\uu t}\right).
\end{eqnarray}

For the covariance between $\newY$ and $\ZZZ$, one obtains from \eqref{eqn:defnnewz}
and \eqref{eqn:addarg210+} that
\begin{equation} \label{eqn:newZYcovarinace}
      \cov(\newY_t,\ZZZ_t) = \sigma^2 \int_0^t  e^{ \uu(s-t) + \Phi(s)-\Phi(t)} \dd s.
\end{equation}

We shall show that for $t\in\TT$
\begin{eqnarray}
     \var\newY_t &=& \frac{\sigma^2}{2\uu}( 1+ O(\eps |\ln \eps|)), \label{eqn:feb20-4.1}
\\
     \var\ZZZ_t &=& \frac{\sigma^2}{2\phi(t)}( 1+ O(\eps |\ln \eps|)),\label{eqn:feb20-4.2}
\\
     \cov(\newY_t,\ZZZ_t) &=& \frac{\sigma^2}{\phi(t)+\uu}( 1+ O(\eps |\ln \eps|)),\label{eqn:feb20-4.3}
\\
     \E[ (\newY_t - \ZZZ_t)^2 ] &=& \sigma^2( O(\sigma^2 \psi) + O( \eps |\ln \eps| ) ),\label{eqn:feb20-4.4}
\end{eqnarray}
and
\be  \label{eqn:feb20-hoelderproperty}
     \E[ (\ZZZ_{t_2}-\ZZZ_{t_1})^2]
     \leq  C \sigma^2 |t_2-t_1|,\qquad t_1,t_2\leq t_\ast, |t_2-t_1|\leq 1.
\ee

{\it Step 2:} Variance analysis: proof of  \eqref{eqn:feb20-4.1} and \eqref{eqn:feb20-4.2}.

Under assumption \eqref{eqn:stretch_Z} we have $\TT\subseteq [t_\ast/2,t_\ast]$.

Relation \eqref{eqn:feb20-4.1} follows now directly from \eqref{eqn:newYvariance},
 as
\[
    e^{-2\uu t}\leq e^{-2\uu t_\ast/2}
    = e^{-\uu d (\br-1) \eps^{-1}} \ll \eps |\ln \eps| \qquad \textrm{for } t\in\TT.
\]

Now we move to the proof of \eqref{eqn:feb20-4.2}.
Recall that under Assumption $\U$  the coefficient function $\phi$ is bounded away from zero on $[0,t_\ast]$,
namely $\phi(t)\geq \kappa_{\min}>0$, $t\in[0,t_\ast]$. This implies that for $0\leq s\leq t$
\[
    \Phi(t)-\Phi(s) = \int_s^t \phi(u) \dd u \geq \kappa_{\min} (t-s).
\]
Therefore, for all $0\leq w\leq t$, we have
\begin{eqnarray*}
      \int_0^{t-w} e^{2(\Phi(s)-\Phi(t))} \dd s &\leq& \int_0^{t-w} e^{2 \kappa_{\min}(s-t)} \dd s
\\
     &\leq& \int_{-\infty}^{t-w} e^{2 \kappa_{\min}(s-t)} \dd s = (2\kappa_{\min})^{-1}e^{-2\kappa_{\min} w}.
\end{eqnarray*}
We use this for $w:=\kappa_{\min}^{-1} |\ln \eps|$ to get that
\begin{equation} \label{eqn:intermediatestepvarinacenewZ}
     \int_0^{t-w} e^{2(\Phi(s)-\Phi(t))} \dd s
     \leq (2\kappa_{\min})^{-1} \eps^2 \ll \eps |\ln \eps|.
\end{equation}
In the essential zone $s\in[t-w,t]$ we use the Taylor expansion of $\Phi$:
\begin{eqnarray}
    \Phi(s)-\Phi(t)&=&\Phi'(t)(s-t) + \frac{1}{2}\Phi''(\tilde s)(s-t)^2 \notag
\\
    &=&\phi(t)(s-t)+\frac{1}{2}\phi'(\tilde s)(s-t)^2 \notag
\\
    &=& \phi(t)(s-t) + \frac{\eps}{2d}\, U'''(\tilde x) (s-t)^2, \label{eqn:aug27-1}
\end{eqnarray}
with  some  $\tilde s\in[s,t]$ and  $\tilde x:= 1+\tfrac{\eps \tilde s}{d} \in[1,\br]$.

Recall that under Assumption $\U$ we have  $|U'''(\tilde x)|\leq K$.
Further, using that $|s-t|\leq w = \kappa_{\min}^{-1} |\ln\eps|$, we obtain
\[
    \Phi(s)-\Phi(t) \leq \phi(t)(s-t) + \frac{K}{2d \, \kappa_{\min}}\, \eps |\ln\eps| (t-s).
\]
It follows that
\begin{eqnarray*}
    \int_{t-w}^t \exp(2(\Phi(s)-\Phi(t))) \dd s
    &\leq& \int_{t-w}^t \exp\left( 2\left(\phi(t)-\frac{K}{2d\, \kappa_{\min}}\, \eps |\ln\eps| \right)(s-t)\right) \dd s
\\
     &\leq& \int_0^\infty \exp\left(- 2\left(\phi(t)-\frac{K}{2d \kappa_{\min}} \, \eps |\ln\eps| \right) u\right) \dd u
\\
     &=&\frac{1}{2(\phi(t)-\frac{K}{2d \kappa_{\min}} \, \eps |\ln\eps| )}
\\
     &\leq & \frac{1}{2\phi(t)} (1+O(\eps |\ln\eps|)),
\end{eqnarray*}
having used again that $U''$ (and so $\phi$) is bounded away from zero.

Putting together the last relation with \eqref{eqn:intermediatestepvarinacenewZ}
and \eqref{eqn:newZvariance} shows the upper bound in \eqref{eqn:feb20-4.2}.

For the lower bound in \eqref{eqn:feb20-4.2}, we argue similarly.
Again we set $w:=\kappa_{\min}^{-1} |\ln \eps|$. Using the Taylor expansion of $\Phi$ from \eqref{eqn:aug27-1},
we obtain for $s\in[t-w,t]$
\[
   \Phi(s)-\Phi(t) \geq \phi(t) (s-t) - \frac{\eps}{d} \, K (s-t)^2
   \geq  \left[\phi(t) + \frac{\eps}{2d} \, K w \right] (s-t).
\]
Using this estimate, we get that
\begin{eqnarray*}
    \int_0^t e^{2(\Phi(s)-\Phi(t))}\dd s &\geq& \int_{t-w}^t e^{-2(t-s) [\phi(t)+\eps w K/(2d)]} \dd s
\\
    &=& \int_0^w e^{-2 u [\phi(t)+\eps w K/(2d)]} \dd u
\\
    &=& \frac{1}{2[\phi(t)+\eps w K/(2d)]} \cdot (1-e^{-2 w [\phi(t)+\eps w K /(2d)]}).
\end{eqnarray*}
The last term is lower bounded by
\begin{eqnarray*}
   \frac{1}{2\phi(t)} \left( 1 - \frac{\eps w K}{2d\phi(t)} \right) \cdot (1-e^{-2 w \kappa_{\min}})
   &\ge& \frac{1}{2\phi(t)} \left( 1 - \frac{\eps |\ln\eps| K}{2d\kappa_{\min}^2} \right) \cdot (1-\eps^2)
\\
    &=& \frac{1}{2\phi(t)} \left( 1 - O(\eps|\ln\eps|)\right).
\end{eqnarray*}
This shows (recall \eqref{eqn:newZvariance}) that
$\var \ZZZ_t \geq  \frac{1}{2\phi(t)} \left( 1 - O(\eps|\ln\eps|)\right)$, as required.

{\it Step 3:} Covariance analysis: proof of \eqref{eqn:feb20-4.3} and \eqref{eqn:feb20-4.4}.

Relation \eqref{eqn:feb20-4.3} follows in the same way as we proved \eqref{eqn:feb20-4.2},
because the integral in \eqref{eqn:newZYcovarinace} has the same structure as the integral
in \eqref{eqn:newZvariance}, the only difference being that
$\Phi(t)$ is replaced by $\Phi(t)+ \uu t$ and all properties used in the above proof carry over.

We now show \eqref{eqn:feb20-4.4}. Using the Taylor expansion of $\phi$, we get that for $t\in\TT$,
\[
   \phi(t) = U''(1+\frac{\eps t}{d}) = \uu+U'''(\tilde x) \, \frac{\eps}{d} \, (t-t_\ast),
\]
with some $\tilde x\in[1,\br]$. Therefore, for $t\in\TT$ we have
\[
    |\phi(t)-\uu| \leq\frac{|U'''(\tilde x)|}{d} \cdot \eps|t-t_\ast|
    \leq C \eps \gamma_1 h^{-1} \sqrt{\psi}
    \sim C \gamma_1  \sigma \sqrt{|\ln \eps|} \to 0,
\]
by the assumption \eqref{eqn:stretch_Z} and Assumption $\U$.
Using this in the following computation, we get:
\begin{eqnarray*}
    &&\E[ (\newY_t-\ZZZ_t)^2]
\\
    &=& \var\newY_t +\var\ZZZ_t - 2 \cov(\newY_t,\ZZZ_t)
\\
    &=& \sigma^2 \left( \frac{1}{2\uu} +\frac{1}{2\phi(t)} - \frac{2}{\uu+\phi(t)} + O(\eps |\ln\eps|)\right)
\\
    &=& \sigma^2 \left( \frac{\phi(t)\,\uu +\phi(t)^2 + \uu^2 + \phi(t)\,\uu - 4 \phi(t)\,\uu}
         {2\uu\phi(t)(\uu+\phi(t))} + O(\eps |\ln\eps|)\right)
\\
     &=& \sigma^2 \left( \frac{(\phi(t)-\uu)^2}{2\uu\phi(t)(\uu+\phi(t))} + O(\eps |\ln\eps|)\right)
\\
    &=& \sigma^2 \left( O((\phi(t)-\uu)^2) + O(\eps |\ln\eps|)\right)
\\
    &=& \sigma^2 \left( O(\sigma^2 \psi) + O(\eps |\ln\eps|)\right),
\end{eqnarray*}
as stated.

{\it Step 4:} Proof of \eqref{eqn:feb20-hoelderproperty}.

Let $0\leq t_1\leq t_2\leq t_\ast$. Then, using \eqref{eqn:defnnewz}, we have
\begin{eqnarray*}
    &&    \E[ (\ZZZ_{t_2}-\ZZZ_{t_1})^2 ]
\\
    &=& \sigma^2 (e^{-\Phi(t_2)} - e^{-\Phi(t_1)})^2 \int_0^{t_1} e^{2\Phi(s)} \dd s
        + \sigma^2 e^{-2\Phi(t_2)} \int_{t_1}^{t_2} e^{2\Phi(s)} \dd s
\\
     &=& \sigma^2  ( e^{\Phi(t_1)-\Phi(t_2)} - 1)^2 \int_0^{t_1} e^{2(\Phi(s)-\Phi(t_1))} \dd s
        + \sigma^2 \int_{t_1}^{t_2} e^{2(\Phi(s)-\Phi(t_2))} \dd s
\\
     &\leq& \sigma^2 (\Phi(t_2)-\Phi(t_1))^2 \int_0^{t_1} e^{2(\Phi(s)-\Phi(t_1))} \sigma^2\dd s
       + \sigma^2 \int_{t_1}^{t_2} e^{0} \dd s
\\
      &\leq& \sigma^2  \Big(\int_{t_1}^{t_2} \phi(u)\dd u\Big)^2 \int_0^{t_1} e^{2 \kappa_{\min}(s-t_1)} \dd s
      + \sigma^2 (t_2-t_1)
\\
      &\leq& \sigma^2 \kappa_{\max}^2 (t_2-t_1)^2 \, \frac{1}{2\kappa_{\min}}
      + \sigma^2 (t_2-t_1)
       \leq C \sigma^2 (t_2-t_1),
\end{eqnarray*}
if $t_2-t_1\leq 1$, as required in  \eqref{eqn:feb20-hoelderproperty}.

{\it Step 5: Probability evaluation.}
Define $\DD_t:=\sigma^{-1} ( \newY_t - \ZZZ_t)$. The claim of the
lemma is to show that
\[
    \P \Big( \sup_{t\in\TT} |\DD_t| > \frac{\theta}{\sqrt{\psi}} \Big) \to 0.
\]
Using the union bound and the fact that the length of $\TT$ is $\gamma_1 h^{-1} \sqrt{\psi}$,
it is sufficient to show that
\begin{equation} \label{eqn:02-10-tbsgaussian}
     h^{-1} \sqrt{\psi} \sup_{I\subseteq \TT, |I|=1}
     \P\Big( \sup_{t\in I} |\DD_t| > \frac{\theta}{\sqrt{\psi}} \Big) \to 0.
\end{equation}
In order to estimate the last probability, we will use
the standard techniques from the theory of Gaussian processes, see
\cite{L} and \cite{MLlecturenotes}.
Essentially, we have to estimate the maximal variance of $(\DD_t)$ as well as
the compactness properties of $I$ w.r.t.\ the distance induced by $(\DD_t)$.

Fix $I\subseteq \TT$ with $|I|=1$ and recall from \eqref{eqn:feb20-4.4} that
\begin{equation} \label{eqn:02-10-boundonmaxv}
      V_I:=\max_{t\in I} \var\DD_t \leq C ( \sigma^2 \psi + \eps |\ln\eps|).
\end{equation}

Let $\bar E$ and $\bar m$ denote the expectation and the median of the random variable
$\sup_{t\in I} \DD_t$, respectively. It is known from the general Gaussian theory that
$\bar m\leq \bar E$, see e.g.\ \cite[Lemma~12.2]{L}.

Assume that we have shown that
\begin{equation} \label{eqn:02-10-finaleqn}
      \bar E \leq \frac{\theta}{2\sqrt{\psi}}.
\end{equation}

Then by concentration principle, see \cite[Theorem~12.2]{L},  we have
\begin{eqnarray*}
      \P( \sup_{t\in I} |\DD_t| > \frac{\theta}{\sqrt{\psi}} )
      &\leq & 2 \, \P( \sup_{t\in I} \DD_t > \frac{\theta}{\sqrt{\psi}} )
\\
      &\leq & 2 \exp\left( -\frac{(\frac{\theta}{\sqrt{\psi}} - \bar m)^2}{2\, V_I} \right)
\\
      &\leq & 2 \exp\left( -\frac{(\frac{\theta}{\sqrt{\psi}} - \bar E)^2}{2\, V_I} \right)
\\
      &\leq & 2 \exp\left( -\frac{\theta^2}{8\,\psi \, V_I} \right)
\\
      &\leq & 2 \exp\left( -\frac{\theta^2}{8\, \psi \, C \, (\sigma^2 \psi + \eps |\ln\eps|)} \right)
\\
      &\leq & 2 \exp\left( -\frac{\theta^2}{16\,\psi\, C\, \max\{ \sigma^2 \psi , \eps |\ln\eps| \} } \right),
\end{eqnarray*}
which yields \eqref{eqn:02-10-tbsgaussian} because on the one hand
(using \eqref{eqn:stretch_Z})
\[
     h^{-1} \sqrt{\psi} \, e^{-c \psi^{-1} \eps^{-1} |\ln \eps|^{-1}}
     \ll \eps^{-1} \sqrt{|\ln \eps|} e^{- c \eps^{-1}|\ln\eps|^{-2} } \to 0,
\]
while on the other hand -- using that we can choose $\sigma^2 \psi^3 < \delta$
due to \eqref{eqn:stretch_Z} -- we obtain
\begin{eqnarray*}
    h^{-1} \sqrt{\psi}\,  e^{-c \sigma^{-2} \psi^{-2}}
    &=& h^{-1} \sqrt{\psi}\,  e^{-c \psi \sigma^{-2} \psi^{-3} }
\\
    &=& h^{-1}  e^{-c \psi\sigma^{-2} \psi^{-3}/2 } \,
    \cdot \sqrt{\psi} \, e^{-c \psi\sigma^{-2} \psi^{-3} /2 }
\\
    &\leq& h^{-1}  e^{-c \psi /(2\delta) } \cdot \sqrt{\psi} \, e^{-c \psi/ (2\delta) }
\\
    &=& h^{-1}  h^{c/(2\delta)} \cdot \sqrt{\psi} e^{-c \psi\, /(2\delta) }
    \to 0,
\end{eqnarray*}
for $\delta$ chosen small enough.

{\it Step 6:} We finally show \eqref{eqn:02-10-finaleqn}.

We shall use the Dudley bound, see \cite[Theorem~14.1]{L},
\begin{equation} \label{eqn:dudley02-10}
     \bar E \leq 4 \sqrt{2} \int_0^{\sqrt{V_I}} (\ln N(\rho))^{1/2} \dd \rho,
\end{equation}
where $N(\rho)$ is the minimal number of $\rho$-balls that is needed to cover $I$
in the process-induced distance
\[
   \Delta(s,t):= \E[ |\DD_t - \DD_s|^2 ]^{1/2}
\]

From the result \eqref{eqn:feb20-hoelderproperty} (and the corresponding result for $\newY$,
which is simple to show) one obtains that
\[
   \E[ |\DD_t - \DD_s|^2 ]^{1/2} \leq C |t-s|^{1/2},
   \qquad  t,s\in\TT,
\]
showing $N(\rho) \leq C' \rho^{-2}$ for $\rho>0$. This implies that the Dudley integral
in \eqref{eqn:dudley02-10} is upper bounded by a constant times $\sqrt{V_I |\ln V_I|}$.
The claim in \eqref{eqn:02-10-finaleqn} follows, if this quantity is of lower order
compared to $\theta/\sqrt{\psi}$, i.e.\ we need to show that $V_I |\ln V_I| \ll \psi^{-1}$.
Taking relation \eqref{eqn:02-10-boundonmaxv} into account, this is obtained from the following
two relations:
\begin{eqnarray}
\label{eqn:psisigma}
    \psi^{-1} \gg \sigma^2 \psi \, |\ln (\sigma^2 \psi)|,
\\
\label{eqn:psieps}
 \psi^{-1} \gg \eps |\ln \eps| |\ln( \eps |\ln \eps| ) |.
\end{eqnarray}

Finally, using \eqref{eqn:stretch_Z}, we obtain
\[
   \sigma^2 \psi\, |\ln (\sigma^2 \psi)| \ll  (\sigma^2\psi)^{2/3} = (\sigma^2\psi^3)^{2/3} \psi^{-4/3}
   \ll \psi^{-4/3}  \ll \psi^{-1},
\]
thus proving \eqref{eqn:psisigma}. Furthermore, from
\[
   \psi = \ln(\sigma/\eps) \leq |\ln \eps|
\]
inequality \eqref{eqn:psieps} follows trivially.
\end{proof}

\subsubsection*{A representation of $\vecZ$}
Let us return to the SDE system
\eqref{eqn:original_phi}
with a scalar function $\phi(\cdot)$, where $\phi(t)=U''(1+\frac{\eps t}{d})$.
We will now connect \eqref{eqn:original_phi} to the scalar processes treated in
Lemma~\ref{lem:winternoteslemma}.

The system \eqref{eqn:original_phi} can be written in the vector form
$$
(Z^1_t,\ldots,Z^{d-1}_t)^\top = \phi(t) \AA (Z^1_t,\ldots,Z^{d-1}_t)^\top \dd t + \sigma \dd \boldsymbol{B}_t,
$$
where the $(d-1)\times(d-1)$-matrix $\AA$ is defined by $\AA_{i,i}=-2$, $\AA_{i,i+1}=\AA_{i+1,i}=-1$, and $\AA_{i,j}=0$ otherwise.\footnote{Recall that $\vecZ$ has $d+1$ components, namely $(Z^0,\ldots,Z^d)$ with the trivial parts $Z^0\equiv 0$ and $Z^d_t=d+\eps t$, while in the last equation we only want to represent the non-trivial components of $\vecZ$.}
Consider a diagonalization of $\AA$ in the form
$\AA=Q^\top D Q$, where $D=\diag(\lambda_1,\ldots,\lambda_{d-1})$ and
$Q$ being a unitary operator. We only need that all eigenvalues $\lambda_j$, $1\leq j \leq d-1$, are negative.

Further, consider the scalar SDEs:
\be \label{eqn:newZdefn2}
     \dd \ZZZ_t^j
     = \lambda_j\phi(t)\ZZZ_t^j\dd t+\sigma \dd B^j_t,\qquad t\geq 0,\qquad \ZZZ_0^j=0,\qquad j=1,\ldots, d,
\ee
with independent Brownian motions $B^j$ and the same scalar function $\phi$ as above.
Note that up to the prefactors $\lambda_j$, these are the processes treated in
Lemma~\ref{lem:winternoteslemma}.
The system of these equations can be rewritten in the vector form with $\vecZZZ_t=(\ZZZ^1_t,\ldots,\ZZZ^{d-1}_t)$
\begin{equation} \label{eqn:newZdefnstar}
     \dd \vecZZZ_t = D\vecZZZ_t \phi(t)\dd t+\sigma \dd \boldsymbol{B}_t,\qquad \vecZZZ_0=0.
\end{equation}

Set
\[
    \generalizedq_t:=(\generalizedq_t^1,\ldots,\generalizedq_t^{d-1})
    :=-\int_0^t \exp( \AA (\Phi_t-\Phi_s) ) \dd s \cdot \boldsymbol{\nu},
\]
and further $\generalizedq_t^0\equiv \generalizedq_t^d\equiv 0$,
where $\boldsymbol{\nu}^j:=j/d$, $j=1,\ldots, d-1$, and $\Phi_t:=\int_0^t \phi(s) \dd s$. It is simple to check that
\[
    \AA \generalizedq_t \phi(t) = \generalizedq_t' + \boldsymbol{\nu}.
\]

This yields the following representation of $(Z_t^1,\ldots,Z_t^{d-1})$
in terms of processes $(\ZZZ^1_t,\ldots,\ZZZ^{d-1}_t)$.

\begin{lem} \label{lem:represenation}   Assume that the $\vecZZZ$ solves  \eqref{eqn:newZdefnstar}.
Then the following is a solution to \eqref{eqn:original_phi}:
\begin{equation}
\label{eqn:represenationZs}
    Z_t^i := \frac{i}{d}( \eps t +d) + \eps \generalizedq_t^i + (Q^\top \vecZZZ)^i_t,\qquad i=1,\ldots,d-1,
\end{equation}
where $(\generalizedq_t^i)$ are bounded deterministic functions defined above.
\end{lem}

\begin{proof}
Note that
\begin{eqnarray*}
    && ( Z_t^{i+1} -2Z^i_t +Z^{i-1}_t) \phi(t) \dd t + \sigma \dd \tilde B_t^i
\\
    &=& \frac{i+1-2i+i-1}{d}(\eps t +d)  \phi(t) \dd t + \eps (\generalizedq_t^{i+1}-2\generalizedq_t^{i}+\generalizedq_t^{i-1})  \phi(t) \dd t
\\
    && + ((Q^\top\vecZZZ)^{i+1}_t -2 (Q^\top\vecZZZ)^i_t+(Q^\top\vecZZZ)^{i-1}_t) \phi(t)  \dd t + \sigma \dd \tilde B_t^i
\\
    &=& \eps (\AA \generalizedq_t \phi(t) )^i   \dd t
    + (\AA Q^\top \vecZZZ_t \phi(t) )^i  \dd t + \sigma \dd \tilde B_t^i
\\
    &=& \eps (\generalizedq_t' + \boldsymbol{\nu} )^i   \dd t
    + (Q^\top D \vecZZZ_t \phi(t) )^i  \dd t + \sigma  \dd (Q^\top B)_t^i
\\
    &=& \eps (\generalizedq_t' + \boldsymbol{\nu} )^i   \dd t
    + (Q^\top  \vecZZZ)_t^i  \dd t
\\
    &=&   \dd Z^i_t,
\end{eqnarray*}
where we used \eqref{eqn:newZdefnstar} in the last but one step. Also the initial condition is verified:
\[
    Z_0^i = i + \eps \generalizedq_0^i + (Q^\top \vecZZZ)^i_0 = i.
\]
Let us finally show that the functions $(\generalizedq_t^i)$ are bounded. Indeed,
let $\mu:=\min |\lambda_j| > 0$. Then
\begin{eqnarray*}
    ||\generalizedq_t||_\infty &\leq &||\generalizedq_t||_2
\\
     &\le& \int_0^t || \exp( \AA (\Phi_t-\Phi_s) ) \boldsymbol{\nu} ||_2 \dd s
\\
     &=& \int_0^t || Q^\top \exp( D (\Phi_t-\Phi_s) ) Q \boldsymbol{\nu} ||_2 \dd s
\\
     &=& \int_0^t || \exp( D (\Phi_t-\Phi_s) ) Q \boldsymbol{\nu} ||_2 \dd s
\\
     &=& \int_0^t \left( \sum_{j=1}^{d-1} \exp( -2\lambda_j (\Phi_t-\Phi_s) ) |
     (Q \boldsymbol{\nu})^j  |^2 \right)^{1/2} \dd s
\\
     &\le& \int_0^t \left( \sum_{i=j}^{d-1} \exp( - 2\mu (\Phi_t-\Phi_s) ) |
     (Q \boldsymbol{\nu})^j  |^2 \right)^{1/2} \dd s
\\
     &=&   \int_0^t \exp( - \mu (\Phi_t-\Phi_s) ) \cdot || Q \boldsymbol{\nu}||_2 \dd s
\\
     &=&   \int_0^t \exp( - \mu \int_s^t \phi(r) \dd r  ) \cdot || \boldsymbol{\nu}||_2 \dd s
\\
     &\le&   \int_0^t \exp( - \mu \int_s^t \kappa_{\min} \dd r  ) \dd s\, || \boldsymbol{\nu}||_2
\\
     &=&  (1-\exp( - \mu t  )) \, \frac{ || \boldsymbol{\nu}||_2}{\mu \kappa_{\min}}.
\end{eqnarray*}
\end{proof}

\begin{rem}
\label{rem:linearcomb}
We proved in Lemma~\ref{lem:winternoteslemma} that the solution to \eqref{eqn:newZdefn} satisfies
\begin{eqnarray*}
    \var \ZZZ^j_t &\leq& C_1 \sigma^2,\quad t\leq t_\ast,
\\
    \E[ (\ZZZ^j_{t_1} - \ZZZ^j_{t_2})^2] &\leq & C_2 \sigma^2 |t_1 -t_2|,\qquad t_1,t_2\leq t_\ast, |t_1-t_2|\leq 1,
\end{eqnarray*}
which carries over to linear combinations of the $\ZZZ^j$, e.g.\
to the processes $(Q^\top \vecZZZ)^i$ in the representation \eqref{eqn:represenationZs}.
\end{rem}

\subsection{Proof of Lemma~\ref{lem:noearlyZ}}
\label{sec:noearlyZ}

Let us fix a break position $i\in\{1,...,d\}$.
Our starting point is a representation from Lemma~\ref{lem:represenation}:
\be \label{eqn:ZDeltaV}
     Z^i_t -Z^{i-1}_t = q_t+ \eps (g^i_t-g^{i-1}_t) + \sum_{j=1}^{d-1} c_{ij} \ZZZ^j_t
      =:  q_t+ \eps \Delta^i_t + V^i_t, \qquad i=0,\ldots,d-1,
\ee
where $\Delta^i_t$ is a bounded deterministic part, and $\ZZZ^j$ are the independent processes from
\eqref{eqn:newZdefn2} (for this proof, independence is irrelevant). Let $D^i:= \sup_{t>0} |\Delta^i_t|$.

The exit condition of $i$-th component
at a time $s$ is now equivalent to
$ q_s+ \eps \Delta^i_t +  V^i_s=\br$;
in other words, $V^i_s= \frac{\eps}{d}(t_\ast-s)- \eps \Delta^i_t$.
A necessary condition for the break is
$V^i_s\geq \frac{\eps}{d}(t_\ast-s)- \eps D^i$.
We may restate it as
\[
  \P(\tau^i_{Z,\br}\leq t) \leq \P\left(\exists s\leq t: V^i_s \geq \frac{\eps}{d}(t_\ast-s)-\eps D^i \right).
\]
\medskip

By Remark~\ref{rem:linearcomb}, the centered Gaussian process $V^i$,
being a linear combination of the $\ZZZ^j$, inherits their following properties:
\begin{eqnarray*}
  \var V^i_t &\le& C_1 \sigma^2, \qquad t\leq t_\ast,
\\
  \E\left(V^i_{t_2}-V^i_{t_1}\right)^2 &\leq&  C_2 \sigma^2 |t_1-t_2|, \qquad t_1,t_2\leq t_\ast, |t_1-t_2|\leq 1,
\end{eqnarray*}
with the constants $C_1,C_2$ depending on dimension $d$ and potential $U$.
By standard arguments of Gaussian process theory these bounds yield
\begin{equation} \label{eqn:GaussLD}
  \P\left( \sup_{s\in[t-1,t]} |V^i_s|\geq \sigma R\right)
  \leq \exp\left(-C_3 R^2\right),
  \qquad R>0, 1\leq t \leq t_\ast,
\end{equation}
with some $C_3=C_3(C_1,C_2)$.
\medskip

By using \eqref{eqn:GaussLD}, it follows that
\begin{eqnarray*}
  \P(\tau^i_{Z,\br}\leq t)
  &\leq& \sum_{k=0}^\infty  \P\left(\sup_{s\in[t-k-1,t-k]} |V^i_s|\geq \frac{\eps}{d}(t_\ast-t+(k-dD^i))\right)
\\
  &\leq& \sum_{k=0}^\infty  \exp\left( - C_3 [\frac{\eps}{d}(t_\ast-t+(k-dD^i))]^2/\sigma^2 \right)
\\
  &\leq&  \exp\left( - C_3 [\frac{\eps}{d}(t_\ast-t)]^2/\sigma^2 \right)
  \sum_{k=0}^\infty    \exp\left( - 2 C_3 \frac{\eps^2}{d^2\sigma^2}(t_\ast-t)(k-dD^i)) \right)
\\
  &=&  \exp\left( - C_3 [\frac{\eps}{d}(t_\ast-t)]^2/\sigma^2 \right)
   \left[ 1 -   \exp\left( - 2 C_3 \frac{\eps^2}{d^2\sigma^2}\, (t_\ast-t) \right)\right]^{-1}
\\ && \times \   \exp\left( 2 C_3 \frac{\eps^2}{d^2 \sigma^2}(t_\ast-t)D^i)) \right) .
\end{eqnarray*}

Letting here $t=t_\ast - \gamma_1 \, \frac{\sigma}{\eps} \sqrt{\ln\left(\frac{\sigma}{\eps}\right)}$,
as in the assertion of Lemma~\ref{lem:noearlyZ}, we obtain
\begin{eqnarray*}
  \P(\tau^i_{Z,\br}\leq t)
  &\leq& \left(\frac{\sigma}{\eps}\right)^{-C_3\gamma_1^2/d^2} \
  \left[  2 C_3 \frac{\gamma_1 \eps}{d^2 \sigma} \sqrt{\ln\left(\frac{\sigma}{\eps}\right)} \, \right]^{-1} (1+o(1))
\\
   &\leq& \left(\frac{\sigma}{\eps}\right)^{1-C_3\gamma_1^2/d^2} \   \frac{d^2}{2 C_3\gamma_1}\
   \left( \ln\left(\frac{\sigma}{\eps}\right)\right)^{-1/2} \ (1+o(1))
   \to 0,
\end{eqnarray*}
if we choose $\gamma_1$ so large that $C_3\gamma_1^2/d^2>1$. This finishes the proof of Lemma~\ref{lem:noearlyZ}.


\subsection{Proof of Lemma~\ref{lem:ZXclose}}
\label{sec:ZXclose}

As above, we set $\mathcal{T}:=[t_\ast-\gamma_1 \frac{\sigma}{\eps} \sqrt{\ln \sigma/\eps},t_\ast]$.

By Lemma~\ref{lem:represenation}, the solutions to \eqref{eqn:original} and \eqref{eqn:original_phi}
can be represented by
\[
    X^i_{\uu,t} = \frac{i}{d} ( \eps t +d) + \eps \generalizedq_t^i + (Q^\top \boldsymbol{\newY})^i_t
\]
and
\[
    Z^i_t = \frac{i}{d} ( \eps t +d) + \eps \tilde\generalizedq_t^i + (Q^\top \boldsymbol{\ZZZ})^i_t,
\]
respectively, where $\boldsymbol{\generalizedq}$ and $\boldsymbol{\tilde\generalizedq}$
are bounded deterministic functions and $\boldsymbol{\newY}$ and $\boldsymbol{\ZZZ}$
are the vectors of the solutions  to \eqref{eqn:newZdefn} with constant prefactor $\uu$ and
varying prefactor $\phi$, respectively.

For the differences, we have
\[
    X^i_{\uu,t} - Z^i_t = \eps( \generalizedq_t^i-\tilde\generalizedq_t^i) + (Q^\top (\newY-\ZZZ) )_t^i,
\]
so that (using the boundedness of $\generalizedq_t$ and $\tilde\generalizedq_t$) for some $K>0$ and all
$t\geq 0$
\begin{eqnarray*}
||\vecX_{\uu,t}-\vecZ_t||_\infty &\le& \eps K + ||Q^\top (\newY-\ZZZ)||_\infty
\\
&\le& \eps K + ||Q^\top (\newY-\ZZZ)||_2
\\
&=& \eps K + ||\newY-\ZZZ||_2
\\
&\le& \eps K + \sqrt{d-1} ||\newY-\ZZZ||_\infty.
\end{eqnarray*}

Fix $\theta>0$. Using that $\eps\ll \sigma/\sqrt{\ln \left(\tfrac{\sigma}{\eps}\right)}$, we have,
as $\eps,sigma\to 0$,
\begin{eqnarray*}
     && \P\left( \sup_{t\in\mathcal{T}}||\vecX_{\uu,t}-\vecZ_t||_\infty
         > \frac{\theta \sigma}{\sqrt{\ln  \left(\tfrac{\sigma}{\eps}\right)}} \right)
\\
     &\le& \P\left( \sqrt{d-1} \sup_{t\in\mathcal{T}}||\newY_t-\ZZZ_t||_\infty
            > \frac{1}{2} \frac{\theta \sigma}{\sqrt{\ln \left(\tfrac{\sigma}{\eps}\right)}} \right)
\\
     &\le& \sum_{i=1}^{d-1} \P\left(  \sup_{t\in\mathcal{T}}|\newY_t^i-\ZZZ_t^i|
            > \frac{(\theta/(2\sqrt{d-1})) \sigma}{\sqrt{\ln \left(\tfrac{\sigma}{\eps}\right)}} \right) \to 0,
\end{eqnarray*}
by Lemma~\ref{lem:winternoteslemma}. This finishes the proof of Lemma~\ref{lem:ZXclose}.


\subsection{Proof of Proposition \ref{prop:vbgeneral}}
\label{sec:proof_prop:vbgeneral}

Let, as above, $\MA$ denote the discrete Laplace operator in
one dimension with $d$ supporting points, i.e.\ the $(d-1)$-dimensional square 
matrix $\MA$ with $\MA_{i,j} = -2$ when $i=j$,  
$\MA_{i,j} = 1$ when $|i-j| = 1$,  
and $\MA_{i,j} = 0$ otherwise.
The largest eigenvalue $\lambda_1$ of $\MA$ is strictly negative, namely 
$\lambda_1 = -2 (1 - \cos(\pi/d))$.

By Assumption $\U$, $U''$ is continuous and strictly positive on $[1,\br]$. 
Therefore, there exist some $r>0$ and $u_2=u_2(r)>0$ such that 
$U''(x) > u_2(r)$ for all $x \in [1, \br+r]$. 
By the continuity of the third derivative of $U$, we have 
$u_3(r) := \sup \{  |U'''(x)|: |x| \leq \br +r \} < \infty$. 
We define the constant
\[
     c_r := c(U,d,r) : = \frac{6 (d-1) u_3(r)}{|\lambda_1| u_2(r)}
\]
and the stopping time
\[
     t_r(\omega) := \inf \{ t \in \mathbb R: \exists i \leq n \text{ with } 
     |\XX_s^i(\omega) - \XX_s^{i-1}(\omega)| \geq \br+r \}
     \wedge t_\ast.
\]

\begin{prop} \label{prop:comparison}
Let $r$, $c_r$ be as above. For all $\omega \in \Omega$ and all $t \leq t_r(\omega)$ we have
\[
S_{t}^\ast(\omega) \leq c_r  \big( (S_{t}^\ast(\omega))^2 + \frac12 (M_{t}^\ast(\omega))^2 \big).
\]
\end{prop}

\begin{proof}
Let $t$ be arbitrary at first.
We define $\vecWW_t = \vecXX_t - \vecZ_t$. Since $\vecZ$ and $\vecXX$ are driven 
by the same Brownian motions, the process $\vecWW$ fulfils
\begin{eqnarray*}
    \dd \WW_t^i &=&  \dd \XX_t^i - \dd Z_t^i
    = U'(\XX^{i+1}_t-\XX^{i}_t) \, \dd t - U'(\XX^{i}_t - \XX^{i-1}_{t}) \, \dd t
\\
     && - \, U''(q_t) (Z_t^{i+1} + Z_t^{-1} - 2 Z_t^i) \, \dd t.
\end{eqnarray*}
By adding and subtracting the term $U''(q_t) (\XX_t^{i+1} + \XX_t^{i-1} - 2 \XX_t^i) \, \dd t$,
and making the definition
\begin{equation}
\label{eqn:defphi}
	(\psi_t(x))^i :=
    U'(x^{i+1} - x^{i}) - U'(x^{i} - x^{i-1}) - U''(q_t) (x^{i+1} + x^{i-1} - 2 x^i ),
\end{equation}
we obtain
\[
   \dd \WW_t^i
   =  (\psi_t(\XX_t))^i + U''(q_t) (\WW_t^{i+1} - \WW_t^{i-1} - 2 \WW_t^i) \dd t
\]
In matrix notation, using the discrete Laplacian, this reads as
\[
    \dd \vecWW_t = \vecpsi_t(\vecXX_t) + \mathcal U''(q_t) \MA \vecWW_t \, \dd t,
\]
and we get
\[
    \vecWW_t = \int_0^t {\rm e}^{\int_s^t U''(q_v) \, d v \MA} \vecpsi_s(\vecXX_s) \, \dd s.
\]
The largest eigenvalue of the matrix ${\rm e}^{\int_s^t U''(q_v) \, d v \MA}$ is bounded above by
${\rm e}^{(t-s) u_2 \lambda_1}$, and thus by the matrix norm inequality we obtain
\begin{eqnarray} \label{eqn:intermediateStep}
   \| \vecZ_t(\omega) - \vecXX_t(\omega) \|_2
   &=& \| \vecWW_t(\omega) \|_2 \leq \int_0^t {\rm e}^{(t-s) u_2 \lambda_1} \|\psi_s(\vecXX_s(\omega))\|_2 \, \dd s
    \nonumber
\\
    &\leq & \frac{1}{u_2 \lambda_1} \sup_{s \leq t} \|\vecpsi_s(\vecXX_s(\omega))\|_2
\end{eqnarray}
for all $t > 0$ and all $\omega \in \Omega$.

When we Taylor expand the first two terms on the right hand side of \eqref{eqn:defphi} around the point $q_t$,
the terms of order one cancel the third term there, and
the second order remainder terms give the estimate
\[
   | (\psi_s(x))^i| \leq \frac12 | U'''(\xi_+) (x^{i+1} - x^i - q_s)^2| +  \frac12
   |U'''(\xi_-) (x^{i} - x^{i-1} - q_s)^2|
\]
where $\xi_+$ lies between $q_s$ and $x^{i+1} - x^i$, and where
$\xi_-$ lies between $q_s$ and $x^{i} - x^{i-1}$.

With $s \leq t_r(\omega)$ and $x^i = \XX_s(\omega)$, the definition of $t_r(\omega)$ and the fact that
$0 \leq q_s \leq \br$ for $s \leq t \leq t_0 \leq t_\ast$ yield $|\xi_{\pm}| < \br + r$.
Therefore,
\[
   |\psi_s(\vecXX_s(\omega))^i| \leq \frac{u_3(r)}{2} \Big(  (\XX^{i+1}_s(\omega) - \XX_s^i(\omega) - q_s)^2
   + (\XX_s^{i}(\omega) - \XX_s^{i-1}(\omega) - q_s)^2 \Big),
\]
and thus
\[
   \|\vecpsi_s(\vecXX_s(\omega))\|_2
   \leq (d-1)  \sum_{i=1}^{d-1} |\psi_s(\XX_s(\omega))^i| \leq  (d-1) u_3(r) \sum_{i=1}^d
   (\XX_s^{i}(\omega) - \XX_s^{i-1}(\omega) - q_s)^2.
\]
for all $s \leq t_r(\omega)$.
By the inequality
\[
    (\XX_s^{i} - \XX_s^{i-1} - q_s)^2
    \leq 3 (Z_s^i - \XX_s^i)^2 + 3 (Z_s^{i-1} - \XX_s^{i-1})^2 + 3(Z_s^i - Z_s^{i-1} - q_s)^2,
\]
and since $Z^0 \equiv \XX^0$ and $Z^d \equiv \XX^d$, we obtain
\[
   \|\vecpsi_s(\vecXX_s(\omega))\|_2 \leq  6 (d-1) u_3(r) \Big( \|\vecXX_s(\omega) - \vecZ_s(\omega) \|_2^2
   + \frac12 \sum_{i=1}^d | Z_s^i - Z_s^{i-1} - q_s|^2 \Big).
\]
Inserting this into the inequality \eqref{eqn:intermediateStep},
we obtain
\[
   \| \vecZ_{\tilde t} - \vecXX_{\tilde t} \|_2
   \leq
   \frac{6 (d-1) u_3(r)}{u_2(r) \lambda_1} \big( (S^\ast_{\tilde t}(\omega))^2
       + \frac12 (M^\ast_{\tilde t}(\omega))^2 \big)
\]
for all ${\tilde t} \leq t_r(\omega)$. The maps 
$\tilde t \mapsto (S^\ast_{\tilde t}(\omega))$ and
$\tilde t \mapsto (M^\ast_{\tilde t}(\omega))$
are monotone increasing. Therefore for $t \leq t_r(\omega)$, we obtain
the result by taking the supremum over $\tilde t \leq t$ on both sides of the above inequality.
\end{proof}

\begin{proof} [ of Proposition \ref{prop:vbgeneral}]
We decompose
\begin{eqnarray}
   \mathbb P(S_t^\ast \geq R) & = & \mathbb P( S_t^\ast \geq R, t \leq t_r, \sqrt{2} c_r M_t^\ast < 1)   \label{eqn:decomp1}
\\
   && + \mathbb P(S_t^\ast \geq R, t > t_r, \sqrt{2} c_r M_t^\ast < 1)  \label{eqn:decomp2}
\\
   && + \mathbb P(S_t^\ast \geq R, \sqrt{2} c_r M_t^\ast \geq 1). \label{eqn:decomp3}
\end{eqnarray}
The term \eqref{eqn:decomp3} is simply estimated by $\mathbb P(\sqrt{2} c_r M_t^\ast \geq 1)$, giving the third term on
the right hand side of the claim.

Turning to the term \eqref{eqn:decomp1}, we will show that for all $\omega \in \Omega$ with
$S_t(\omega) \geq R$, $t \leq t_r(\omega)$ and $\sqrt{2} c_r M_t^\ast(\omega) < 1$, the inequality
\begin{equation}
\label{eqn:the first term}
     M_t^\ast(\omega)^2 \geq \frac{S_t^\ast(\omega)}{2 c_r} \geq \frac{R}{2 c_r}
\end{equation}
holds, which then gives the first term on the right hand side of the claim. To see \eqref{eqn:the first term},
recall first that by Proposition \ref{prop:comparison} we have
\[
     c_r (S_s^\ast(\omega))^2 - S_s^\ast(\omega) + \frac12 c_r (M_s^\ast(\omega))^2 \geq 0
\]
for all $s \leq t_r(\omega)$. Since $t \leq t_r(\omega)$, this inequality holds for all $s \leq t$. Let $M > 0$.
The equation $y = c_r (y^2 + M^2/2)$ has two nonnegative solutions if and only if $c_r^2 M^2 < 1/2$, and in this case
the smaller one of those is given by
\[
   \frac{1 - \sqrt{1 - 2 c_r^2 M^2}}{2 c_r} < 2 c_r M^2.
\]
By the condition $\sqrt{2} c_r M_t^\ast(\omega) < 1$ and monotonicity 
we have $\sqrt{2} c_r M_s^\ast(\omega) < 1$ for all $s \leq t$;
by the above considerations (with $M = M_s^\ast(\omega)$), 
we find that for all $s \leq t$ and all $\omega$ fulfilling the
relevant conditions in \eqref{eqn:decomp1}, the value of $S_s^\ast(\omega)$ 
can not be in the interval between the two solutions of the quadratic equation 
for any $s \leq t$. Since the function $s \mapsto S_s^\ast(\omega)$ is continuous 
and has the value $0$ for $s=0$, it therefore has to stay to the left of the smaller
root, and is therefore for $s=t$ bounded by $2 c_r M_t^\ast(\omega)^2$. 
We thus arrive at \eqref{eqn:the first term}.

Finally we discuss the term \eqref{eqn:decomp2}. 
The considerations of the previous paragraph still apply, but only up
to $s = t_r(\omega) < t$. Therefore while we do not have \eqref{eqn:the first term}, 
we still know that
\[
    S_{t_r(\omega)}^\ast(\omega) \leq 2 c_r M_{t_r(\omega)}^\ast(\omega)
\]
for all $\omega$ relevant to \eqref{eqn:decomp2}.
Since  $t_r(\omega)< t \leq t_\ast$, by definition of $t_r(\omega)$, 
there is at least one $i \leq d$ with 
$|\XX_{t_r(\omega)}^i(\omega) - \XX_{t_r(\omega)}^{i-1}(\omega)| = \br+r$, 
and we get
\begin{eqnarray*}
    \br+r &=& |\XX_{t_r(\omega)}^i(\omega) - \XX_{t_r(\omega)}^{i-1}(\omega)|
    \leq
     |\XX_{t_r(\omega)}^i(\omega) - Z_{t_r(\omega)}^{i}(\omega)|
\\
      && + |Z_{t_r(\omega)}^i(\omega) - Z_{t_r(\omega)}^{i-1}(\omega)|
         + |\XX_{t_r(\omega)}^{i-1}(\omega) - Z_{t_r(\omega)}^{i-1}(\omega)|
\\
      &\leq& |Z_{t_r(\omega)}^i(\omega) - Z_{t_r(\omega)}^{i-1}(\omega)| 
      + 2 S_{t_r(\omega)}^\ast(\omega)
\\
     &\leq&  |Z_{t_r(\omega)}^i(\omega) - Z_{t_r(\omega)}^{i-1}(\omega)| 
     + 4 c_r (M_{t_r(\omega)}^\ast(\omega))^2.
\end{eqnarray*}
Therefore,
\begin{eqnarray*}
     M_t^\ast(\omega) & \geq & |Z^i_{t_r(\omega)}(\omega) - Z^{i-1}_{t_r(\omega)}(\omega) 
     - q_{t_r(\omega)}|
\\
     &\geq&  |Z^i_{t_r(\omega)}(\omega) - Z^{i-1}_{t_r(\omega)}(\omega)| - \br
\\
     &\geq& r - 4 c_r (M_{t_r(\omega)}^\ast(\omega))^2 \geq r - 4 c_r (M_{t}^\ast(\omega))^2.
\end{eqnarray*}
Therefore, $M_t^\ast(\omega) + 4 c_r (M_t^\ast)(\omega)^2 \geq r$ for all $\omega$ relevant to \eqref{eqn:decomp2},
and we obtain the second term on the right hand side of the claim.
\end{proof}


\subsection{Proof of Proposition \ref{prop:vbspecial}}
\label{sec:proof_prop:vbspecial}

By Proposition \ref{prop:vbgeneral}, it is enough to prove
that for $t=t_\ast=\tfrac{d(\br-1)}{\eps}$ the right hand side of \eqref{eqn:vbgeneral} tends to zero,
as $\eps,\sigma\to 0$.

Under assumption \eqref{eqn:stretch_vb}
we have $\delta\to 0$. Therefore, the first probability in the
right hand side of \eqref{eqn:vbgeneral} dominates two others. It remains to prove that
for every $\theta>0$ we have
\[
   \lim_{\eps,\sigma\to 0} \P(M_{t_\ast}^*\geq \sqrt{\delta}/C)=0.
\]
Since $\theta>0$ is arbitrary, we may drop $C$ here. Furthermore, by using the union bound, it is sufficient to prove
that
\[
   \lim_{\eps,\sigma\to 0} \P( \sup_{0\leq s\leq t_\ast}
   |Z_s^i-Z_s^{i-1}- q_s|
   \geq \sqrt{\delta})=0,
   \qquad i=1,...,d.
\]
We fix $i$. By representation \eqref{eqn:ZDeltaV} we have
\begin{eqnarray*}
   \P( \sup_{0\leq s\leq t_\ast} |Z_s^i-Z_s^{i-1}- q_s| \geq \sqrt{\delta})
   &\leq&
    \P( \sup_{0\leq s\leq t_\ast} |V_s^i| \geq \sqrt{\delta}-\eps D_i)
\\
   &\leq&
   (t_\ast+1) \max_{0 \leq k\leq [t_\ast]}
    \P( \sup_{k\leq s\leq k+1} |V_s^i| \geq \sqrt{\delta}-\eps D_i).
\end{eqnarray*}
Under  \eqref{eqn:stretch_vb} we have
\[
   \frac{\eps}{\delta} = \frac{\eps}{\theta\sigma} \ \sqrt{\ln\left(\frac{\sigma}{\eps}\right)} \to 0.
\]
Hence, $\eps\ll\delta\ll\sqrt{\delta}$, and eventually  $\sqrt{\delta}-\eps D_i\geq  \sqrt{\delta}/2$.
By using \eqref{eqn:GaussLD}, we obtain
\[
    \P( \sup_{0\leq s\leq t_\ast} |Z_s^i-Z_s^{i-1}- q_s| \geq \sqrt{\delta})
   \leq
    (t_\ast+1) \exp\left\{- \frac{C_3 \delta}{4\sigma^2}\right\}.
\]
Notice that $t_\ast\approx \eps^{-1}$, while under assumption \eqref{eqn:stretch_vb}
\[
 \sigma^{-2}\delta= \theta \sigma^{-1}/ \sqrt{\ln\left(\frac{\sigma}{\eps}\right)}
 \geq  \theta \sigma^{-1}/ \sqrt{|\ln\eps|}
 = \theta \left(\sigma^2  |\ln\eps|^3 \right)^{-1/2}  |\ln\eps| \gg |\ln\eps|,
\]
which completes the proof of Proposition~\ref{prop:vbspecial}.

\bigskip
{\bf Acknowledgement.} This research was supported by the co-ordinated
grants of DFG (AU370/7) and RFBR (20-51-12004).


\end{document}